\documentclass[11pt, a4paper, leqno]{scrartcl}
\usepackage{amsmath}
\usepackage{amssymb}
\usepackage{mathrsfs}
\usepackage{amsthm}  
\usepackage{microtype} 
\usepackage{enumitem}
\usepackage{xcolor}
\usepackage{comment}
\usepackage{authblk}
\usepackage[hidelinks]{hyperref}
\usepackage[utf8]{inputenc}
\usepackage[T1]{fontenc}

\theoremstyle{definition}
\newtheorem{definition}{Definition}
\newtheorem{theorem}[definition]{Theorem}
\newtheorem{corollary}[definition]{Corollary}
\newtheorem{lemma}[definition]{Lemma}
\newtheorem{proposition}[definition]{Proposition}

\newtheorem{fact}[definition]{Fact}
\newtheorem{claim}[definition]{Claim}
\newtheorem{Proof}[definition]{Proof}

\newcommand{\cL}{\mathcal L}

\newcommand{\sln}{\mathcal L^{\text{s},n}}

\newcommand{\ULS}{\text{ULST}}
\newcommand{\SULS}{\text{SULST}}

\newcommand{\crit}{\text{crit}}
\newcommand{\cof}{\text{cof}}
\newcommand{\WVP}{\text{WVP}}
\newcommand{\VP}{\text{VP}}

\newcommand{\dom}{\text{dom}}

\newcommand{\ZFC}{\text{ZFC}}
\newcommand{\comp}{\text{comp}}

\newcommand{\Ord}{\text{Ord}}

\newcommand{\ran}{\text{ran}}

\newcommand{\LST}{\text{LST}}

\newcommand\blfootnote[1]{%
	\begingroup
	\renewcommand\thefootnote{}\footnote{#1}%
	\addtocounter{footnote}{-1}%
	\endgroup
}

\title{Model theory of class-sized logics}
\author{Jonathan Osinski\thanks{Fachbereich Mathematik, Universität Hamburg, Hamburg, Germany.} \thanks{\emph{Current address}: Czech Academy of Sciences, Institute of Computer Sciences, Pod Vodárenskou věží 2, 182 07 Prague, Czech Republic. \\Email: osinski@cs.cas.cz} and   Trevor Wilson\thanks{Department of Mathematics, Miami University, Oxford, Ohio, 45056, USA. \\ Email: twilson@miamioh.edu}}

\begin{document}

\maketitle

\begin{abstract}
	\noindent \textbf{Abstract.} We study compactness and Löwenheim-Skolem properties of fragments of the class-sized logic $\cL_{\infty \infty}$ and of class-sized versions of second-order and sort logics. In these fragments, certain combinations of infinitary quantifiers and boolean connectives are banned. While model-theoretic properties fail for unrestricted class logics, this drastically changes in our more restricted setting. We show that model-theoretic properties of class logics characterise a wide array of large cardinals, and that some of them can even be obtained in $\ZFC$. In particular, we give a characterisation of \emph{Weak Vopěnka's Principle} and \emph{Ord is Woodin} by downwards Löwenheim-Skolem properties, and a characterisation of \emph{Shelah cardinals} by a compactness property of class-sized logics. We further strengthen many known results about properties of set-sized logics by studying how they transfer to class-sized extensions.  
\end{abstract}

	\section{Introduction}
	\noindent In\blfootnote{The results of this article also appear in the first author's PhD thesis \cite[Chapter 5]{osinskiPhD}.} many discussions of model theory of strong logics $\cL$, i.e., of extensions of first-order logic, it has become somewhat standard to reserve the term \emph{logic} to set-sized languages, i.e., to languages for which, given any set of non-logical symbols $\tau$, the collection $\cL[\tau]$ of $\cL$-sentences over $\tau$ is again a set (cf., e.g., \cite{mak1985, boney2024model, osinskiPhD, lücke2024weak} as examples for this naming practice). As a consequence, some languages classically studied in abstract model theory are not considered to be a logic in this sense, e.g., the infinitary language $\cL_{\infty \infty}$ allowing for conjunctions and disjunctions over arbitrarily sized sets of sentences, and quantification over arbitrarily sized sets of variables. Let us call $\cL$ a \emph{class-sized logic} or \emph{class logic} when $\cL[\tau]$, the collection of $\cL$-sentences over $\tau$, is a proper class for some $\tau$. The reason that class-sized logics are often not considered is that many of the usual properties studied in abstract model theory are easily seen to fail for $\cL_{\infty \infty}$ (cf. Proposition \ref{prop:HanfLS}), arguably the easiest example of a class-sized logic available. Hence, already $\cL_{\infty \infty}$ does not have much of an interesting model theory. 
	
	In this article, we show that this drastically changes when considering class-sized logics which are not closed under negation. We study class-sized subsystems of $\cL_{\infty \infty}$, as well as of stronger languages like class-sized versions of second-order and sort logics. We show that assertions about properties of many of these systems are equivalent to the existence of large cardinals, and that some of them even have model-theoretic properties obtainable in ZFC. In some cases, these results strengthen known theorems about properties of set-sized logics. In others, these results give completely new characterizations of large cardinals.
	
	As a motivation, let us mention some background about the relation of model theory of strong logics and large cardinals. How large cardinals are characterized by model-theoretic properties of strong logics is a classical realm of study since Tarski \cite{tar1962} asked the question how the Compactness Theorem for first-order logic carries over to infinitary logics of the form $\cL_{\kappa \kappa}$ for an uncountable regular cardinal $\kappa$. The answer led to the introduction of weakly and strongly compact cardinals. Until the 1980s, many prominent large cardinal axioms were found to be equivalent to statements about model-theoretic properties of strong logics (cf., e.g., \cite{mag1971, ben1978, mak1985}). Recently, this type of research had a resurgence, and many more connections of large cardinals to compactness properties of logics, generalizing the Compactness Theorem, as well as to (upward and downward) Löwenheim-Skolem properties, generalizing the Löwenheim-Skolem Theorems, were discovered (cf. \cite{mag2011, bagaria2016symbiosis, bon2020, hay2022, boney2024model, holy2024, osinskiPhD, galeotti2025bounded, gitman2025upward, lücke2024weak, boney2025, osinski2024Henkin}).

	The article is structured as follows. In Section \ref{sec:Prelim} we recall some basic definitions and background from abstract model theory, introduce the class-sized logics we will work with, and collect definitions and background on large cardinal notions we will consider. In Section \ref{sec:ClassComp} we study how switching from the set-sized versions of first-order logic, infinitary logics, second-order logics, and sort logics, to certain class extensions impacts their respective compactness and upward Löwenheim-Skolem properties. We show that in some cases, this switch comes with a jump in large cardinal strength of the considered property, while in others, the strength stays the same. In Section \ref{sec:ClassWVP}, we study downwards Löwenheim-Skolem properties of class extensions of sort logic. In this vein, we obtain the first known characterization of \emph{$\Pi_n$-strong cardinals}, \emph{Weak Vopěnka's Principle}, and \emph{$\Ord$ is Woodin} by Löwenheim-Skolem properties. The connection between $\Pi_n$-strong cardinals, Weak Vopěnka's Principle, and \emph{$\Ord$ is Woodin}  was recently uncovered by work of Wilson \cite{wilson2022large} and Bagaria and Wilson \cite{bagaria2023weak}. That the usual \emph{Vopěnka's Principle} has several model-theoretic characterizations is well-known. For instance, it is equivalent to axiom schemas stating that all logics have certain compactness properties and downwards Löwenheim-Skolem properties, respectively. Whether the same is true for Weak Vopěnka's Principle has been open. We give a positive answer by showing that  Weak Vopěnka's Principle is equivalent to the simultaneous existence of certain Löwenheim-Skolem numbers for a large collection of class logics. Independently, also Boney and Osinski \cite{boney2025}, and, Holy, Lücke, and Müller \cite{holy2024}, showed that Weak Vopěnka's Principle can be characterized by model-theoretic properties of strong logics. Both approaches use compactness principles though. By our characterization by Löwenheim-Skolem properties, we thus strengthen the model-theoretic analogy between Vopěnka's Principle and Weak Vopěnka's Principle in also carrying over to the latter types of properties. Finally, in Section \ref{sec:Shelah} we provide the (to our best knowledge) first known characterization of \emph{Shelah cardinals} by properties of strong logics, using a property akin to weak compactness of a class logic.
	 
	\section{Preliminaries}\label{sec:Prelim}
	\subsection{Logics}
	Let us start by making precise what we understand by a \emph{logic}, as an abstract object of study. The definition is the standard one taken from \cite[Chapter II, Definition 1.1.1]{bar1985}. Note that we do \emph{not} follow the naming practice described in the beginning, in which a logic has to have set-many sentences over set-sized vocabularies. We will consider \emph{many-sorted} vocabularies, as one of the logics we will consider is \emph{sort logic}, whose semantics is naturally expressed in this context. As in standard definitions in first-order model theory, a \emph{(many-sorted) vocabulary} $\tau$ is a set consisting of finitely many sort symbols and an arbitrary amount of finitary relation, function, and constant symbols. We denote the set of sort symbols in $\tau$ by $\text{s}(\tau)$. We demand $\text{s}(\tau) \neq \emptyset$. A \emph{$\tau$-structure} is a tuple consisting of non-empty sets $A_s$ for every sort symbol $s \in \tau$ and interpretations for the relation, function, and constant symbols. We write $A$ for $\bigcup \{A_s \colon s \in \text{s}(\tau)\}$. If $\tau$ and $\sigma$ are vocabularies, a bijection $f: \tau \rightarrow \sigma$ is called a \emph{renaming} if it respects arities and restricts to bijections of the respective subsets of sort, relation, function, and constant symbols in $\tau$ and $\sigma$. If $\mathcal A$ is a $\tau$-structure and $f: \tau \rightarrow \sigma$ a renaming, then $f$ induces a $\sigma$-structure we will denote by $f(\mathcal A)$ on $A$ by letting $A_{f(s)} = A_s$ for $s \in \text{s}(\tau)$ and $f(r)^{f(\mathcal A)} = r^\mathcal A$ for every $r \in \tau \setminus \text{s}(\tau)$. An isomorphism between $\tau$-structures $\mathcal A$ and $\mathcal B$ is a bijection $A \rightarrow B$ respecting the sorts, relations, functions, and constants in the obvious way. We write $\mathcal A \cong \mathcal B$ if there is an isomorphism between $\mathcal A$ and $\mathcal B$.
	\begin{definition}\label{def:abstr:log}
		A \textit{logic} $\cL$ is a pair consisting of a definable class function that maps every vocabulary $\tau$ to a class $\cL[\tau]$, called the \textit{class of $\cL$-sentences over $\tau$}, and a definable class relation $\models_{\cL}$, called the \textit{satisfaction relation of $\cL$}, such that:
		\begin{enumerate}
			\item[(i)] If $\mathcal A \models_{\cL} \varphi$, then $\varphi \in \cL[\tau]$ for some vocabulary $\tau$ and $\mathcal A$ is a $\tau$-structure. In this case we say that \emph{$\mathcal A$ is a model of $\varphi$}. If the logic is clear from context, we also simply write $\mathcal A \models \varphi$. 
			\item[(ii)] If $\sigma \subseteq \tau$ for vocabularies $\tau$ and $\sigma$, then $\cL[\sigma] \subseteq \cL[\tau]$.
			\item[(iii)] If $\mathcal A \cong \mathcal B$ for $\tau$-structures $\mathcal A$ and $\mathcal B$, then for every $\varphi \in \cL[\tau]$:
			\[
			\mathcal A \models_{\cL} \varphi \text{ iff } \mathcal B \models_{\cL} \varphi.
			\]
			\item[(iv)] If $\varphi \in \cL[\sigma]$ and $\sigma \subseteq \tau$ for vocabularies $\tau$ and $\sigma$, then for every $\tau$-structure $\mathcal A$:
			\[
			\mathcal A \models_{\cL} \varphi \text{ iff } (\mathcal A \upharpoonright \sigma) \models_{\cL} \varphi.
			\]
			\item[(v)] If $\tau$ is a vocabulary and $f: \tau \rightarrow \sigma$ is a renaming, then for every $\varphi \in \cL[\tau]$ there is a $\psi \in \cL[\sigma]$ such that for every $\tau$-structure $\mathcal A$:
			\[
			\mathcal A \models_\cL \varphi \text{ iff } f(\mathcal A) \models_\cL \psi.
			\]
		\end{enumerate}
		If there is a vocabulary $\tau$ for which $\cL[\tau]$ forms a proper class, we call $\cL$ a \emph{class-sized logic}. If the following condition is fulfilled, we call $\cL$ a \emph{set-sized logic}.
		\begin{itemize}
			\item[(vi)] For every vocabulary $\tau$, $\cL[\tau]$ is a set, and $\cL$ has an \emph{occurrence number} $\text{o}(\cL)$, which is the smallest cardinal $\kappa$ such that for every vocabulary $\tau$ and every $\varphi \in \cL[\tau]$, there is some $\tau_0 \subseteq \tau$ with $|\tau_0| < \kappa$ and $\varphi \in \cL[\tau_0]$.  
		\end{itemize}
	\end{definition}
	Let us now define the logics we will consider. Intuitively, we study fragments of class logics like $\cL_{\infty \infty}$ that drop some of the boolean connectives or quantifiers. For instance, the logic $\cL (\wedge^\infty, \exists^\infty, \forall^\infty)$ is allowed to form conjunctions over arbitrarily large sets of sentences and to existentially and universally quantify over arbitrarily large sets of first-order variables, but it can only form \emph{finitary} disjunctions. To make sure disjunctions over arbitrary sets are not expressable, $\cL (\wedge^\infty, \exists^\infty, \forall^\infty)$ is \emph{not} closed under negation. Technically, this is achieved by demanding that infinitary boolean connectives and quantifiers cannot occur inside of negations.
	
	Let us make this more precise and give a general definition that allows to vary our notation to denote many of the logics we are interested in. To unify notation, below we use expressions like $\kappa = \infty$. When writing for a class $S$ that $|S| < \infty$, we simply mean that $S$ is a set. 
	\begin{definition}\label{def:classlogic}
		For $i \in \{0, \dots, 3\}$, let $\kappa_i$ be a regular cardinal or $\kappa_i = \infty$. The collection of formulas of the logic $\cL(\wedge^{\kappa_0}, \vee^{\kappa_1}, \exists^{\kappa_2}, \forall^{\kappa_3})$ over a vocabulary $\tau$ consists of the following formulas:
		\begin{itemize}
			\item All formulas of first-order logic over $\tau$.
			\item If $S \subseteq \cL(\wedge^{\kappa_0}, \vee^{\kappa_1}, \exists^{\kappa_2}, \forall^{\kappa_3})$ and $|S| < \kappa_0$, then $\bigwedge S \in \cL(\wedge^{\kappa_0}, \vee^{\kappa_1}, \exists^{\kappa_2}, \forall^{\kappa_3}).$
			\item If $S \subseteq \cL(\wedge^{\kappa_0}, \vee^{\kappa_1}, \exists^{\kappa_2}, \forall^{\kappa_3})$ and $|S| < \kappa_1$, then $\bigvee S \in \cL(\wedge^{\kappa_0}, \vee^{\kappa_1}, \exists^{\kappa_2}, \forall^{\kappa_3})$.
			\item If $\varphi \in \cL(\wedge^{\kappa_0}, \vee^{\kappa_1}, \exists^{\kappa_2}, \forall^{\kappa_3})$ and $W$ is a set of first-order variables with $|W| < \kappa_2$, then $\exists W \varphi \in \cL(\wedge^{\kappa_0}, \vee^{\kappa_1}, \exists^{\kappa_2}, \forall^{\kappa_3})$. 
			\item If $\varphi \in \cL(\wedge^{\kappa_0}, \vee^{\kappa_1}, \exists^{\kappa_2}, \forall^{\kappa_3})$ and $W$ is a set of first-order variables with $|W| < \kappa_3$, then $\forall W \varphi \in \cL(\wedge^{\kappa_0}, \vee^{\kappa_1}, \exists^{\kappa_2}, \forall^{\kappa_3})$. 
		\end{itemize}
		The semantics of $\cL(\wedge^{\kappa_0}, \vee^{\kappa_1}, \exists^{\kappa_2}, \forall^{\kappa_3})$ are exactly as expected, but to fix notation for usage in proofs we will make a few conventions explicit. A \emph{variable assignment} for a $\tau$-structure $\mathcal A$ is a map $f: X \to A$ where $X$ is a set of variables. We assume variables $x$ to come with a fixed sort $s_x \in \text{s}(\tau)$ and that $f(x) \in A_{s_x}$. If $S$ is a set of variables, then a variable assignment $g$ for $\mathcal A$ is called an \emph{$S$-variant of $f$} if $S \subseteq \dom(g)$ and $f$ and $g$ agree on $X \setminus S$. If $S = \{x\}$ is a singleton, we also call $g$ an \emph{$x$-variant}. If $X$ contains the set of free variables of some $\chi \in \cL(\wedge^{\kappa_0}, \vee^{\kappa_1}, \exists^{\kappa_2}, \forall^{\kappa_3})$, we write
		\[
		\mathcal A \models_{\cL(\wedge^{\kappa_0}, \vee^{\kappa_1}, \exists^{\kappa_2}, \forall^{\kappa_3})} \chi[f]
		\]
		to indicate that $\mathcal A$ is a model of $\varphi$ when its free variables are interpreted according to $f$. The semantics of the quantifiers can be expressed using these conventions, e.g.,
		\begin{quote}
			$\mathcal A \models_{\cL(\wedge^{\kappa_0}, \vee^{\kappa_1}, \exists^{\kappa_2}, \forall^{\kappa_3})} \exists W \varphi[f]$ iff there is a $W$-variant $g$ of $f$ such that $\mathcal A \models_{\cL(\wedge^{\kappa_0}, \vee^{\kappa_1}, \exists^{\kappa_2}, \forall^{\kappa_3})} \varphi[g]$. 
		\end{quote} 
	\end{definition}
	Note that in $\cL(\wedge^{\kappa_0}, \vee^{\kappa_1}, \exists^{\kappa_2}, \forall^{\kappa_3})$, negations can only be applied to finitary first-order formulas. Alternatively, we could have restricted attention to formulas in negation normal form, i.e., negation can only be applied to atomic formulas. By varying the $\kappa_i$ in the above definition, we can denote several set- and class-sized logics. For instance, $\cL(\wedge^{\omega}, \vee^{\omega}, \exists^{\omega}, \forall^{\omega})$ is first-order logic, for $\kappa$ any regular cardinal $\cL(\wedge^{\kappa}, \vee^{\kappa}, \exists^{\kappa}, \forall^{\kappa})$ is the usual infinitary logic $\cL_{\kappa \kappa}$, and $\cL(\wedge^{\infty}, \vee^{\infty}, \exists^{\infty}, \forall^{\infty})$ is the class logic $\cL_{\infty \infty}$. By this notation, we are flexible in denoting logics that forbid some infinitary logical connectives while permitting others. For instance, the logic $\cL^* = \cL(\wedge^{\infty}, \vee^{\omega}, \exists^{\omega}, \forall^{\omega})$ has arbitrarily large conjunctions, but only finitary disjunctions and quantifiers. If $\kappa_i = \omega$ for some $i$, we will omit the corresponding logical symbol when denoting the logic. E.g., for the just mentioned logic $\cL^*$ we will simply write $\cL(\wedge^\infty)$ according to this convention. Note that if $\kappa_i = \infty$ for some $i$, then $\cL(\wedge^{\kappa_0}, \vee^{\kappa_1}, \exists^{\kappa_2}, \forall^{\kappa_3})$ is a class-sized logic. We analogously define fragments of class-versions of second-order logic. We write $\cL^2$ for the usual finitary second-order logic. The definition of its class version $\cL^2(\wedge^{\kappa_0}, \vee^{\kappa_1}, \exists^{\kappa_2}, \forall^{\kappa_3})$ is analogous to Definition \ref{def:classlogic}, just that we allow all formulas of $\cL^2$, in the quantifier cases we additionally allow sets of second-order variables, and we let variable assignments $f$ be able to take second-order variables $X$ and take values $f(X) \subseteq A^n$.
	
	Finally, we consider class extensions of \emph{sort logics}, a collection of logics exceeding second-order logic in expressive strength which were introduced by Väänänen \cite{vää1979}. For their motivation and intuition about sort logics, cf. \cite{vää2014}. Their main feature are \emph{sort quantifiers} $\widetilde \exists X$ and $\widetilde \forall X$. They take second-order variables $X$ and a formula $\widetilde \exists X \varphi (X)$ is satisfied by some structure $\mathcal A$ if $\mathcal A$ can be expanded by additional universes with a relation $X$ between them that satisfies $\varphi$. For definability of truth reasons, sort logics $\sln$ are graded by the natural numbers $n$ into a hierarchy of increasing strength. Here, $\sln$ allows for $n$ alternations of existential and universal sort quantifiers. We will work with infinite versions of sort logics, called $\sln_{\kappa \omega}$ for some regular cardinal $\kappa$, which allow for infinite conjunctions and disjunctions of size $< \kappa$. If $\kappa = \omega$, we simply write $\sln$ for $\sln_{\omega \omega}$. As we will not need more properties of sort logics than the folklore ones summarized below in Fact \ref{factSL}, we refer the reader to \cite[Section 1.2.4]{osinskiPhD} for a precise constructions of sort logic's syntax and semantics. To state the properties we need, we follow \cite{bag2012} in denoting by $C^{(n)}$ the club class of ordinals $\alpha$ such that $V_\alpha$ is a $\Sigma_n$-elementary substructure of the universe.
	\begin{fact}\label{factSL}
		\begin{enumerate}
			\item[(1)] For every $n \geq 1$, there is a sentence $\Phi^{(n)} \in \sln[\{\in\}]$ such that $(M, E) \models_{\sln} \Phi^{(n)}$ iff there is some ordinal $\alpha \in C^{(n)}$ with $(M,E) \cong (V_\alpha, \in)$.
			\item[(2)] Let $\alpha \in C^{(n)}$ and $\tau \in V_\alpha$ be a vocabulary. Then for every  $\tau$-structure $\mathcal A \in V_\alpha$ and every $\varphi \in \sln \cap V_\alpha$:
			\[
			\mathcal A \models_{\sln} \varphi \text{ iff } V_\alpha \models ``\mathcal A \models_{\sln} \varphi".
			\] 
		\end{enumerate}
	\end{fact}
	For a proof of Fact \ref{factSL} consider, e.g., \cite[Corollary 1.2.17 and Corollary 1.2.21]{osinskiPhD}. Note that item (1) is akin to the following result by Magidor.
	
	\begin{fact}[Magidor \cite{mag1971}] \label{fact:MagPhi}
		There is a sentence of second-order logic $\cL^2$, known as \emph{Magidor's $\Phi$}, such that $(M,E) \models_{\cL^2} \Phi$ iff there is some limit ordinal $\alpha$ with $(M, E) \cong (V_\alpha, \in)$.
	\end{fact}

	Now let us introduce the class extension of sort logics we will work with.
	\begin{definition}
		Let $n \geq 1$ and $\kappa$ be a regular cardinal. For $i \in \{0, \dots, 3\}$, let $\kappa_i$ be a regular cardinal $\geq \kappa$ or $\kappa_i = \infty$. The collection of formulas of the logic $\sln_{\kappa \omega}(\wedge^{\kappa_0}, \vee^{\kappa_1}, \exists^{\kappa_2}, \forall^{\kappa_3})$ over a vocabulary $\tau$ consists of the following formulas:
		\begin{itemize}
			\item All formulas of $\sln_{\kappa \omega}$ over $\tau$.
			\item If $S \subseteq \sln_{\kappa \omega}(\wedge^{\kappa_0}, \vee^{\kappa_1}, \exists^{\kappa_2}, \forall^{\kappa_3})$ and $|S| < \kappa_0$, then $\bigwedge S \in \sln_{\kappa \omega}(\wedge^{\kappa_0}, \vee^{\kappa_1}, \exists^{\kappa_2}, \forall^{\kappa_3}).$
			\item If $S \subseteq \sln_{\kappa \omega}(\wedge^{\kappa_0}, \vee^{\kappa_1}, \exists^{\kappa_2}, \forall^{\kappa_3})$ and $|S| < \kappa_1$, then $\bigvee S \in \sln_{\kappa \omega}(\wedge^{\kappa_0}, \vee^{\kappa_1}, \exists^{\kappa_2}, \forall^{\kappa_3})$.
			\item If $\varphi \in \sln_{\kappa \omega}(\wedge^{\kappa_0}, \vee^{\kappa_1}, \exists^{\kappa_2}, \forall^{\kappa_3})$ and $W$ is a set of first- or of second-order variables with $|W| < \kappa_2$, then $\exists W \varphi \in \sln_{\kappa \omega}(\wedge^{\kappa_0}, \vee^{\kappa_1}, \exists^{\kappa_2}, \forall^{\kappa_3})$. 
			\item If $\varphi \in \sln_{\kappa \omega}(\wedge^{\kappa_0}, \vee^{\kappa_1}, \exists^{\kappa_2}, \forall^{\kappa_3})$ and $W$ is a set of first- or of second-order variables with $|W| < \kappa_3$, then $\forall W \varphi \in \sln_{\kappa \omega}(\wedge^{\kappa_0}, \vee^{\kappa_1}, \exists^{\kappa_2}, \forall^{\kappa_3})$. 
		\end{itemize}
	\end{definition}
	Thus, $\sln_{\kappa \omega}(\wedge^{\kappa_0}, \vee^{\kappa_1}, \exists^{\kappa_2}, \forall^{\kappa_3})$ is an expansion of $\sln_{\kappa \omega}$ by conjunctions, disjunctions, and first- and second-order existential or universal quantifiers of size specified by the $\kappa_i$. Note that we restrict usage of the added conjunctions and disjunctions of size $> \kappa$ and of the infinitary quantifiers to take place \emph{outside} the scope of sort quantifiers. I.e., if $\widetilde \exists X \varphi \in \sln_{\kappa \omega}(\wedge^{\kappa_0}, \vee^{\kappa_1}, \exists^{\kappa_2}, \forall^{\kappa_3})$, then we already have $\varphi \in \sln_{\kappa \omega}$. As before, if $\kappa_i = \omega$, we simply drop the corresponding clause, and so write expressions like $\sln_{\kappa \omega}(\wedge^\infty)$ for $\sln_{\kappa \omega}(\wedge^{\kappa_0}, \vee^{\omega}, \exists^{\omega}, \forall^{\omega})$.  
	\subsection{Properties of logics}
	We mention the standard model-theoretic properties of logics considered in the literature and studied below. 
	
	For a logic $\cL$ and a cardinal $\kappa$, a set $T \subseteq \cL$ is called \emph{${<} \kappa$-satisfiable} iff every $S \in \mathcal P_\kappa T$ has a model. The \emph{compactness number} $\comp(\cL)$ of $\cL$ is the smallest cardinal $\kappa$ such that any ${<}\kappa$-satisfiable set $T \subseteq \cL$ has a model.
	
	We also say that $\kappa$ is \emph{a} compactness number of $\cL$ if every ${<}\kappa$-satisfiable set $T \subseteq \cL$ has a model (while not necessarily requiring that $\kappa$ is the smallest such cardinal). For the notions introduced below we follow the analogous convention.
	
	The \emph{weak compactness number} of $\cL$ is the smallest cardinal $\kappa$ such that any ${<}\kappa$-satisfiable set $T \subseteq \cL$ of size $|T| = \kappa$ has a model.
	
	The \emph{Löwenheim-Skolem number} $\text{LS}(\cL)$ of $\cL$ is the smallest cardinal $\kappa$ such that any sentence of $\cL$ has a model of size $< \kappa$.
	
	The \emph{Löwenheim-Skolem-Tarski number} $\LST(\cL)$ is the smallest cardinal $\kappa$ such that if $\tau$ is a vocabulary of size $|\tau| < \kappa$ and $\varphi \in \cL$, then any model $\mathcal A \models_\cL \varphi$ has a substructure $\mathcal B \subseteq \mathcal A$ of size $|B| < \kappa$ with $\mathcal B \models_\cL \varphi$.
	
	The \emph{Hanf number} of $\cL$ is the smallest cardinal $\kappa$ such any satisfiable sentence of $\cL$ with a model $\mathcal A$ of size $|A| \geq \kappa$ has arbitrarily large models.
	
	All the above are classically studied notions. Recently, Galeotti, Khomskii, and Väänänen \cite{galeotti2025bounded} introduced the \emph{upward Löwenheim-Skolem-Tarski number} $\ULS(\cL)$, as the smallest cardinal $\kappa$ such that for all $\varphi \in \cL$ and $\mathcal A \models_\cL \varphi$ of size $|A|\geq \kappa$, there are arbitrarily large superstructures $\mathcal B \supseteq \mathcal A$ with $\mathcal B \models_\cL \varphi$.
	
	This notion got strengthened by Gitman and Osinski \cite{gitman2025upward} to the \emph{strong ULST number} $\SULS(\cL)$, as the smallest cardinal $\kappa$ such that for all sets $T \subseteq \cL$ and $\mathcal A \models_\cL T$ of size $|A| \geq \kappa$, there are arbitrarily large superstructures $\mathcal B \supseteq \mathcal A$ with $\mathcal B \models_\cL T$.
	
	It is folklore that for any set-sized logic $\cL$, it is a theorem of $\ZFC$ that the Löwenheim-Skolem number and the Hanf number of $\cL$ exist (cf., e.g., \cite[Chapter II, Theorem 6.1.4]{bar1985} for a proof for the Hanf number case; the LS number can be treated similarly). The other properties mentioned can have large cardinal strength though. For instance, Magidor showed that $\comp(\cL^2)$ is the smallest extendible cardinal, and that $\LST(\cL^2)$ is the smallest supercompact cardinal (cf. \cite{mag1971}). Results by Fuchino and Sakai \cite{fuchino2022weakly} and by Lücke \cite{lücke2024weak} show that weak compactness cardinals for $\cL^2$ in $L$ give rise to so-called \emph{weakly extendible cardinals}. Gitman and Osinski \cite{gitman2025upward} and Boney and Osinski \cite{boney2025} gave two independent proofs that both $\ULS(\cL^2)$ and $\SULS(\cL^2)$ are the smallest extendible cardinal. The existence of (weak) compactness numbers, LST numbers, and (S)ULST numbers of logics can therefore have large cardinal strength.
	
	It is easy to see that for many of the classically studied class-sized logics, none of these cardinals can consistently exist.
	
	\begin{proposition}\label{prop:HanfLS}
		The logic $\cL_{\infty \infty}$ has neither a Hanf number, nor an LS number.
		\begin{proof}
			Notice that for any cardinal $\kappa$, the sentence
			\[
			\varphi_{\geq \kappa} = \exists (x_i \colon i < \kappa) \bigwedge_{i < j < \kappa} x_i \neq x_j.
			\]
			holds in a structure $\mathcal A$ iff $|A| \geq \kappa$. As $\varphi_{\geq \kappa} \in \cL_{\infty \infty}$ for all cardinals $\kappa$, $\cL_{\infty \infty}$ cannot have an LS number.
			
			\medskip Consider further the sentence
			\[
			\varphi_{< \kappa} = \forall (x_i \colon i < \kappa) \bigvee_{i < j < \kappa} x_i = x_j,
			\]
			which holds in a structure $\mathcal A$ iff $|A| < \kappa$. Because $\varphi_{< \kappa} \in \cL_{\infty \infty}$ for all cardinals $\kappa$, $\cL_{\infty \infty}$ cannot have a Hanf number.
		\end{proof}
	\end{proposition}	
	As the existence of a Hanf number is implied by the existence of the compactness, ULST, and SULST numbers, respectively, and the existence of the LS number is implied by the existence of the LST number, also none of these cardinals can exist for $\cL_{\infty \infty}$. The weak compactness number of $\cL_{\infty \infty}$ cannot exist either, as for any cardinal $\kappa$, the theory $T = \{\varphi_{< \kappa}\} \cup \{c_i \neq c_j \colon i < j < \kappa\}$ for distinct constant symbols $(c_i)_{i < \kappa}$ witnesses that $\kappa$ is not a weak compactness number of $\cL_{\infty \infty}$. 
	
	Note that in the above proof, it is the combination of infinite disjunctions with infinite universal quantification that renders the existence of $\comp(\cL_{\infty \infty})$ impossible. We will see that simply banning this combination allows for the existence of compactness numbers of class-sized logics. For instance, the logic $\cL(\wedge^\infty, \exists^\infty, \forall^\infty)$ has compactness number $\omega$ (cf. Theorem \ref{thm:compclassFOL}). Similarly, banning the combination of infinite conjunctions and existential quantifiers will allow for class logics that \emph{have} downward Löwenheim-Skolem properties.
	
	That, consistently up to large cardinals, class-sized logics can have some interesting model theory was (to our best knowledge) first observed by Wilson. For instance, he obtained the following results, which were presented in several talks, e.g., in \cite{wil2022b}, but remain unpublished.
	\begin{theorem}[Wilson]\label{thm:WilsonClass}
		Let $\kappa$ be a cardinal.
		\begin{enumerate}
			\item[(i)] $\kappa$ is the smallest supercompact cardinal iff $\kappa = \LST(\cL(\vee^\infty, \forall^\infty, \exists^\infty))$.
			\item[(ii)] $\kappa$ is huge with target $\lambda$ iff for every vocabulary $\tau$ of size $< \kappa$ and every $\varphi \in \cL(\vee^\infty, \forall^\infty, \exists^\infty)[\tau]$, if $\mathcal A$ is a $\tau$-structure of size $|A| = \lambda$ and with $\mathcal A \models \varphi$, then there is a substructure $\mathcal B \subseteq \mathcal A$ such that $|B| = \kappa$ and $\mathcal B \models \varphi$. 
		\end{enumerate} 
	\end{theorem}

	\subsection{Large cardinal notions}
	In this section we will give an overview of the large cardinal notions we consider and some of the background results we will refer to. 
	
	A cardinal $\kappa$ is \emph{extendible} if for every $\alpha > \kappa$ there is a some $\beta$ with an elementary embedding $j: V_\alpha \to V_\beta$ such that $\crit(j) = \kappa$ and $j(\kappa) > \alpha$. Equivalently, we might drop the condition that $j(\kappa) > \alpha$ (cf. \cite[Proposition 23.15]{kan}).
	
	The cardinal $\kappa$ is \emph{$C^{(n)}$-extendible} if for every $\alpha > \kappa$ with $\alpha \in C^{(n)}$ there is some $\beta \in C^{(n)}$ with an elementary embedding $j: V_\alpha \to V_\beta$ such that  $\crit(j) = \kappa$ and $j(\kappa) > \alpha$. Again, we may equivalently drop $j(\kappa) > \alpha$.\footnote{The original definition by Bagaria (cf. \cite{bag2012}) demands that for every $\alpha > \kappa$ there is some $j: V_\alpha \to V_\beta$ with $\crit(j) = \kappa$, $j(\kappa) > \alpha$, and $j(\kappa) \in C^{(n)}$. That the presented definition is equivalent to this was shown, e.g., in \cite[Theorem 15]{gitman2019CnExt} and in \cite[Corollary 3.5]{tsaprounis2018c}. The argument from \cite[Proposition 23.15]{kan} that $j(\kappa) > \alpha$ can be dropped then carries over.} With rising $n$, the existence of a $C^{(n)}$-extendible cardinal gains consistency strength (cf. \cite{bag2012}). 
	
	\emph{Vopěnka's Principle} ($\VP$) is the axiom schema stating for every class $\mathcal C$ of structures in a shared vocabulary that there are distinct $\mathcal A, \mathcal B \in \mathcal C$ with an elementary embedding between them. $\VP$ is connected to $C^{(n)}$-extendible cardinals, as well as to properties of logics.
	\begin{fact}\label{fact:VP}
		$\VP$ is equivalent to the following axiom schemas.
		\begin{enumerate}
			\item[(i)] (Bagaria \cite{bag2012}) For every $n$, there is a $C^{(n)}$-extendible cardinal.
			\item[(ii)] (Makowsky \cite{mak1985}) Every set-sized logic has a compactness number.
			\item[(iii)] (Stavi \cite{mag2011}) Every set-sized logic has an LST number.
		\end{enumerate}
	\end{fact}
	Inspired by equivalences of $\VP$ to statements from category theory, Adámek, Rosický, and Trnková \cite{adamek1988} introduced \emph{Weak Vopěnka's Principle} ($\WVP$). While they use the class theory GBC for its formulation, we adopt the $\ZFC$ perspective from \cite{bagaria2023weak}. $\WVP$ is then the axiom schema positing that if $(G_\alpha \colon \alpha \in \text{Ord})$ is a sequence of graphs, which is definable possibly using set parameters, such that for every $\alpha \leq \beta$ there is exactly one homomorphism $G_\beta \rightarrow G_\alpha$, then there are ordinals $\alpha < \beta$ and a homomorphism $G_\alpha \rightarrow G_\beta$.\footnote{In category-theoretic terms: ``The category $\Ord^{\text{op}}$ with the reverse inclusion on ordinals as morphisms does not fully embed into the category $\text{Gra}$ of Graphs".}
	
	Bagaria and Wilson \cite{bagaria2023weak} provided an analysis of $\WVP$ in terms of large cardinals. For this, a cardinal $\kappa$ is called \emph{$\Pi_n$-strong} if for every ordinal $\lambda$ and every $A$, which is $\Pi_n$ definable without parameters, there is a transitive class $M$ and an elementary embedding $j: V \to M$ such that $\crit(j) = \kappa$, $j(\kappa) > \lambda$, $V_\lambda \subseteq M$, and $A \cap V_\lambda \subseteq A^M$. With rising $n$, the existence of a $\Pi_n$-strong cardinal gains consistency strength (cf. \cite[p. 164]{bagaria2023weak}).
	
	When writing $\WVP(\Pi_n)$ for the axiom that restricts the statement $\WVP$ to sequences of graphs definable by a $\Pi_n$ formula, Bagaria and Wilson showed:
	\begin{fact}[{Bagaria \& Wilson \cite[Theorem 5.11]{bagaria2023weak}}]
		The following are equivalent for $n \geq 1$:
		\begin{enumerate}\label{fact:PinWVP}
			\item[(1)] $\WVP(\Pi_n)$.
			\item[(2)] There is a $\Pi_n$-strong cardinal. 
		\end{enumerate}
	\end{fact} 
	
	$\Pi_n$-strong cardinals are a variation of the following classical notion. For a class $A$ (definable \emph{with} possible set parameters), a cardinal $\kappa$ is \emph{$A$-strong} if for every ordinal $\lambda$, there is an elementary embedding $j: V \to M$ such that $\crit(j) = \kappa$, $j(\kappa) > \lambda$, $V_\lambda \subseteq M$, and $j(A \cap V_\kappa) \cap V_\lambda = A \cap V_\lambda$. The axiom schema stating that for every class $A$ there is an $A$-strong cardinal is known as \emph{Ord is Woodin}.
	
	Bagaria and Wilson showed the following.
	\begin{fact}[Bagaria \& Wilson \cite{bagaria2023weak}]\label{fact:BagariaWilson}
		$\WVP$ is equivalent to the following axiom schemas.
		\begin{enumerate}
			\item[(1)] For every $n$, there is a $\Pi_n$-strong cardinal.
			\item[(2)] For every $n$, there is a proper class of $\Pi_n$-strong cardinals.
			\item[(3)] Ord is Woodin.
		\end{enumerate}
	\end{fact}
	That $\WVP$ is equivalent to ``Ord is Woodin" was first shown by Wilson in a GBC setting in  \cite{wilson2022large}.
	
	\section{Compactness of class logics}\label{sec:ClassComp}
	\subsection{Class versions of first-order logic}
	In this section, we will show, in $\ZFC$, that the compactness number of $\cL(\wedge^\infty, \exists^\infty,\forall^\infty)$ is $\omega$. Note that this means that switching from first-order logic to this class logic does have no effect on the compactness number. The proof of this theorem is a refinement of the ultraproduct proof of first-order logic's Compactness Theorem. To obtain this, we show that the relevant usage of \L os' Theorem carries over to $\cL({\wedge^\infty, \exists^\infty, \forall^\infty})$, and afterwards derive our theorem. 
	\begin{lemma}\label{lem:infLos}
		Suppose $\tau$ is a vocabulary, $U$ is an ultrafilter over some set $I$, and for each $i \in I$, there is a $\tau$-structure $M_i$. Let $\varphi$ be a formula of $\cL(\wedge^\infty, \exists^\infty,\forall^\infty)$ over $\tau$ with free variables among a set $X$, and for $i \in I$, let $f_i: X \rightarrow M_i$ be variable assignments. Consider $f: X \rightarrow \prod_{i \in I} M_i /U$, the variable assignment into the ultraproduct of the $M_i$ given by $x \mapsto [i \mapsto f_i(x)]_U$. Then if $\{i \in I \colon M_i \models \varphi[f_i]\} \in U$, then $M = \prod_{i\in I}M_i/U \models \varphi[f]$. 
		\begin{proof}
			We proceed by induction on the structure of $\varphi$. The cases in which $\varphi$ is an atomic formula, or a negation of an atomic formula simply follow by the definition of the ultraproduct. The case in which $\varphi$ is a finitary disjunction is simply the usual step from \L os' Theorem. 
			
			\medskip
			
			So suppose that $S$ is some set of formulas and $\varphi = \bigwedge S$ is an infinite conjunction such that $\{i \in I \colon M_i \models \varphi[f_i]\} \in U$. Let $\psi \in S$. Then $U \ni \{ i \in I \colon M_i \models \varphi[f_i]\} \subseteq \{i \in I \colon M_i \models \psi[f_i]\}$ and so the latter is in $U$. By induction hypothesis, therefore $M \models \psi[f]$. 
			
			\medskip
			
			And now let us consider $\varphi = \exists W \psi$, where $W$ is a set of variables of any size and assume for $U$-many $i$, that we have $M_i \models \exists W \psi[f_i]$, where the $f_i$ are assignments on the set $X$ of variables of $\varphi$. Then for $U$-many $i$ there is a $W$-variant $g_i$ of $f_i$ with $M_i \models \psi[g_i]$. Define $g$ as the function that sends $v \in X$ to $[i \mapsto g_i(v)]_U$. By induction hypothesis it follows that $M \models \psi[g]$. Now if $v \in X \setminus W$, then $g(v) = [i \mapsto g_i(v)]_U = [i \mapsto f_i(v)]_U = f(v)$, because the $g_i$ are $W$-variants of the $f_i$. This shows that $g$ is a $W$-variant of $f$ and therefore $M \models \exists W \psi[f]$. 
			
			\medskip 
			
			And finally, let $\varphi = \forall W \psi$ and assume for $U$-many $i$ that $M_i \models \forall W \psi[f_i]$. For those $U$-many $i$ thus for all $W$-variants $g_i$ of $f_i$ it holds that $M_i \models \psi[g_i]$. Let $g$ be any $W$-variant of $f$. We want to show that $M \models \psi[g]$. Note that by induction hypothesis it is sufficient to show that for $U$-many $i$ we have 
			\begin{equation*}
			\tag{$\ast$} M_i \models \psi[v \mapsto h_v(i)],
			\end{equation*}
			where $h_v: I \rightarrow \bigcup_{i \in I} M_i$ is the function representing $g(v)$, i.e. with $[h_v]_U = g(v)$. Now if $v \in X \setminus W$, then $[h_v]_U = g(v) = f(v) = [i \mapsto f_i(v)]_U$, because $g$ is a $W$-variant of $f$. So for $v \in X \setminus W$ we can without loss of generality assume that $h_v$ is given by the function $i \mapsto f_i(v)$. Then the functions defined by $v \mapsto h_v(i)$ are $W$-variants of the $f_i$. In particular, for $U$-many $i$, we get $M_i \models \psi[v \mapsto h_v(i)]$, so $(\ast)$ holds and we are done. 
		\end{proof}
	\end{lemma}
	
	Recall that a filter $F$ over $\mathcal P_\omega X$ is called \emph{fine} iff for any $x \in X$, $\{s \in \mathcal P_\omega X \colon x \in s\} \in F$. For every non-empty set $X$ there is a fine ultrafilter over $\mathcal P_\omega X$. It is straightforward to check that for any set $X$ there is a fine filter $F$ over $\mathcal P_\omega X$, defined by:
	\[
	Y \in F \text{ iff } Y\subseteq \mathcal P_\omega X \text{ and } \exists x_1, \dots, x_n \in X \colon \{s \in \mathcal P_\omega X \colon x_1, \dots, x_n \in s\} \subseteq Y
	\]
	It is a well-known theorem by Tarski that any filter can be expanded to an ultrafilter over the same set (cf., e.g., \cite[Theorem 7.5]{jec}). Expanding $F$ as above to an ultrafilter $U$ over $\mathcal P_\omega X$ results in a fine ultrafilter.
	\begin{theorem}\label{thm:compclassFOL}
		$\comp(\cL(\wedge^\infty, \exists^\infty, \forall^\infty)) = \omega$.
		\begin{proof}
			Let $T \subseteq \cL(\wedge^\infty, \exists^\infty, \forall^\infty)$ be finitely satisfiable. Then for every $s \in \mathcal P_\omega T$ there is a model $M_s \models s$. Let $U$ be any fine ultrafilter over $\mathcal P_\omega T$. Then for every $\varphi \in T$, the set $\{s \in \mathcal P_\omega T \colon \varphi \in s \} \in U$. But $\{s \in \mathcal P_\omega T \colon \varphi \in s \} \subseteq \{s \in \mathcal P_\omega T \colon M_s \models \varphi\}$, and so the latter is in $U$. Therefore by Lemma \ref{lem:infLos}, $\prod_{s \in \mathcal P_\omega T} M_s/U \models \varphi$ and so this ultraproduct is a model of $T$. 
		\end{proof}
	\end{theorem}
	\begin{corollary}
		The weak compactness number, the Hanf number, the ULST number, and the SULST number of $\cL(\wedge^\infty, \exists^\infty, \forall^\infty)$ are all $\omega$.
		\begin{proof}
			All of these numbers are bounded by the compactness number. For the weak compactness number, this follows simply from the definition. For the others, one can give a simple argument analogous to the usual proof of the upward Löwenheim-Skolem Theorem from the Compactness Theorem for first-order logic (cf., e.g., \cite[Proposition 3.1]{gitman2025upward}).
		\end{proof}
	\end{corollary}
	
	\subsection{Class versions of infinitary logics}
	In the last section we saw that adding arbitrary conjunctions, existential- and universal quantifiers to first-order logic does not change its compactness properties. In this section we will see that a similar assertion is true for infinitary first-order logic $\cL_{\kappa \kappa}$. Recall that a regular uncountable cardinal $\kappa$ is strongly compact iff $\comp(\cL_{\kappa \kappa}) = \kappa$. Further, this is equivalent to the assertion that every $\kappa$-complete filter over any set can be expanded to a $\kappa$-complete ultrafilter (cf., e.g., \cite[Proposition 4.1]{kan}). Completeness of an ultrafilter lets us extend Lemma \ref{lem:infLos} to cover proper class expansions of $\cL_{\kappa \kappa}$ in the to be expected way, expanding the usual \L os-like theorem for $\cL_{\kappa \kappa}$. 
	\begin{lemma}\label{lem:stronglyInfLos}
		Suppose $\tau$ is a vocabulary, $\kappa$ is a regular cardinal, $U$ is a $\kappa$-complete ultrafilter over some set $I$, and for each $i \in I$, there is a $\tau$-structure $M_i$. Let $\varphi$ be a formula of $\cL(\wedge^\infty, \exists^\infty,\forall^\infty, \vee^\kappa)$ over $\tau$ with free variables among a set $X$, and for $i \in I$, let $f_i: X \rightarrow M_i$ be variable assignments. Consider $f: X \rightarrow \prod_{i \in I} M_i /U$, the variable assignment into the ultraproduct of the $M_i$ given by $x \mapsto [i \mapsto f_i(x)]_U$. Then if $\{i \in I \colon M_i \models \varphi[f_i]\} \in U$, then $M = \prod_{i\in I}M_i/U \models \varphi[f]$. 
		\begin{proof}
			Exactly like in the proof of Lemma \ref{lem:infLos}, proceed by induction on $\varphi$. The only case not covered there is $\varphi = \bigvee S$ for some set of formulas $S$ of size $< \kappa$. This is taken care of by $\kappa$-completeness: If $\{i \in I \colon M_i \models \bigvee S [f_i] \} \in U$, note that $\{i \in I \colon M_i \models \bigvee S [f_i] \} = \bigcup_{\psi \in S} \{i \in I \colon M_i \models \psi[f_i] \}$. By $\kappa$-completeness of $U$, there is thus a fixed $\psi \in S$ such that $\{i \in I \colon M_i \models \psi[f_i] \} \in U$. By induction hypothesis therefore $\prod_{i \in I} M_i/U \models \psi[f]$ and hence $\prod_{i\in I} M_i/U \models \bigvee S[f]$. 
		\end{proof}
	\end{lemma}
	\begin{theorem}
		Let $\kappa$ be a regular uncountable cardinal. Then $\kappa$ is strongly compact iff $\comp(\cL(\wedge^\infty, \exists^\infty, \forall^\infty, \vee^\kappa)) = \kappa$.
		\begin{proof}
			Clearly, if $\comp(\cL(\wedge^\infty, \exists^\infty, \forall^\infty, \vee^\kappa)) = \kappa$, then $\comp(\cL_{\kappa \kappa}) = \kappa$, and thus $\kappa$ is strongly compact. On, the other hand, if $\kappa$ is strongly compact and $T \subseteq \cL(\wedge^\infty, \exists^\infty, \forall^\infty, \vee^\kappa)$ is ${<} \kappa$-satisfiable, then for any $s \in \mathcal P_\kappa T$ there is a model $M_s \models s$. Note that
			\[
			Y \in F \text{ iff } Y \subseteq \mathcal P_\kappa T \text{ and } \exists t \in \mathcal P_\kappa T: \{s \in \mathcal P_\kappa T \colon t \subseteq s  \} \subseteq Y
			\]
			defines a $\kappa$-complete, fine filter over $\mathcal P_\kappa T$. By strong compactness of $\kappa$, there is a $\kappa$-complete, fine ultrafilter $U$ over $\mathcal P_\kappa T$ extending $F$. Then Lemma \ref{lem:stronglyInfLos} implies that $\prod_{s \in \mathcal P_\kappa T} M_s/U \models T$. 
		\end{proof}
	\end{theorem}
	
	\subsection{Class versions of second-order logic and sort logics}\label{subsec:SOLSL}
	We already mentioned that the compactness number, the ULST, and the SULST number of second-order logic are the first extendible cardinal, respectively.  Again, as for first-order logic, we show that switching to appropriate proper class versions of second-order logic does not increase the compactness number. Contrastingly though, we show that the switch drastically increases the Hanf and weak compactness numbers. Recall that $\ZFC$ proves the existence of Hanf numbers of any set-sized strong logic. Further, Lücke \cite{lücke2024weak} showed that the axiom schema stating the existence of weak compactness numbers of any set-sized logic is equivalent to \emph{``Ord is faint"}. In particular, the existence of the weak compactness number of $\cL^2$ is weaker than, say, a measurable. Our theorem therefore shows that the existence of a weak compactness number, and of a Hanf number of $\cL^2(\wedge^\infty, \exists^\infty, \forall^\infty)$, are both much stronger than the existence of the respective numbers for $\cL^2$.
	\begin{theorem}\label{thm:classSOL}
		The following are equivalent for a cardinal $\kappa$:
		\begin{enumerate}
			\item[(1)] $\kappa$ is the smallest extendible cardinal.
			\item[(2)] $\kappa$ is the compactness number of $\mathcal L^2$.
			\item[(3)] $\kappa$ is the compactness number of $\cL^2(\wedge^\infty, \exists^\infty, \forall^\infty)$.
			\item[(4)] $\kappa$ is the weak compactness number of $\mathcal L^2(\wedge^\infty, \exists^\infty, \forall^\infty)$.
			\item[(5)] $\kappa$ is the Hanf number of $\mathcal L^2(\wedge^\infty, \exists^\infty, \forall^\infty)$.
			\item[(6)] $\kappa$ is the ULST number of $\mathcal L^2(\wedge^\infty, \exists^\infty, \forall^\infty)$.
			\item[(7)] $\kappa$ is the SULST number of $\mathcal L^2(\wedge^\infty, \exists^\infty, \forall^\infty)$. 
		\end{enumerate}
		\begin{proof}
			The first two items are equivalent by Magidor's theorem \cite{mag1971}. If $\lambda$ is a compactness number of $\cL^2(\wedge^\infty, \exists^\infty, \forall^\infty)$ it is an upper bound on the weak compactness number, ULST number, and $\SULS$ number of the same logic. Further, if $\kappa$ is an $\SULS$ number, then it is a $\ULS$ number, and if it is a $\ULS$ number, then it is a Hanf number. So it is sufficient to show $(1) \Rightarrow (3)$, $(4) \Rightarrow (1)$, and $(5) \Rightarrow (1)$. 
			
			\medskip For these implications it suffices to show that if $\kappa$ is the first extendible cardinal, then any ${<}\kappa$-satisfiable subset $T \subseteq\cL^2(\wedge^\infty, \exists^\infty, \forall^\infty)$ has a model, and that if $\kappa$ is either the weak compactness cardinal or the Hanf number of $\cL^2(\wedge^\infty, \exists^\infty, \forall^\infty)$, then there is an extendible cardinal $\leq \kappa$. Then by minimality of the properties we are considering, our theorem follows.
			
			\medskip So first assume that $\kappa$ is the weak compactness cardinal of  $\cL^2(\wedge^\infty, \exists^\infty, \forall^\infty)$. We aim to show that there is an extendible cardinal $\lambda \leq \kappa$. For this purpose let $\alpha > \kappa$ be any limit ordinal. We want to show that there is a $\delta_\alpha$ and an elementary embedding $j_\alpha: V_\alpha \rightarrow V_{\delta_\alpha}$ with $\crit(j_\alpha) \leq \kappa$. If we showed this, as there are class many ordinals above $\kappa$ but at most $\kappa$ many cardinals below $\kappa$, there is then a fixed $\lambda \leq \kappa$ that is the critical point of $j_\alpha$ for unboundedly many $\alpha$. This $\lambda$ is then extendible. To this end, let $\gamma = |V_\alpha|$ and $(a_i \colon i < \gamma)$ be an enumeration of $V_\alpha$. Assume without loss of generality that $a_0 = \kappa$. Take a constant symbol $c$ and consider the following sentence $\psi$ of $\cL^2(\wedge^\infty, \exists^\infty, \forall^\infty)$:
			\[
			\psi = \exists (b_i \colon i < \gamma) \bigwedge_{\substack{n \in \omega, \varphi(x_1, \dots, x_n) \in \mathcal L_{\omega \omega}[\{\in \}], \\ i_1, \dots, i_n < \gamma, \\ V_\alpha \models \varphi(a_{i_1}, \dots, a_{i_n})}} \varphi(b_{i_1}, \dots, b_{i_n}) \wedge b_0 = c.
			\]
			Note that the big conjunction above codes the elementary diagram of $(V_\alpha, \in)$ into a single sentence of $\cL^2(\wedge^\infty, \exists^\infty, \forall^\infty)$. Hence, if $M \models \psi$, then a sequence $(b_i^M \colon i < \gamma)$ witnessing this gives rise to an elementary embedding $j: V_\alpha \rightarrow M$ by letting $a_i \mapsto b_i^M$. Note that $j(\kappa) = b_0^M = c^M$. Take additional constants $c_\alpha$ for $\alpha \leq \kappa$ and consider the theory
			\[
			T = \{\psi \} \cup \{\Phi \} \cup \{c_i < c_j < c \colon i < j \leq \kappa \} \cup\{``c_i \text{ is an ordinal"}
			\colon i \leq \kappa \}.
			\] 
			Here $\Phi$ is Magidor's $\Phi$ (cf. Fact \ref{fact:MagPhi}). If $T$ is satisfiable, say by $(M, E^M, c^M, c_i^M)_{i \leq \kappa} \models T$, we have that without loss of generality it is of the form $(V_\delta, \in, c^{V_\delta}, c_i^{V_\delta})_{i \leq \kappa}$ for some $\delta$ by the usage of $\Phi$. Further, because $\psi \in T$, there is an elementary embedding $j: V_\alpha \rightarrow V_\delta$ as described above. In particular, $j(\kappa) = c^{V_\delta}$. Now because $c_i^{V_\delta} < c_j^{V_\delta} < c^{V_\delta} = j(\kappa)$ for all $i < j \leq \kappa$, we have that $j(\kappa) > \kappa$. In particular, $j$ has a critical point $\leq \kappa$. So to show satisfiability of $T$ suffices. Now clearly $T$ has size exactly $\kappa$ and is ${<} \kappa$-satisfiable (by $V_\alpha$). By our assumption, it is therefore satisfiable. 
			
			\medskip Now let $\kappa$ be the Hanf number of $\cL^2(\wedge^\infty, \exists^\infty, \forall^\infty)$. Again, we want to show that there is an extendible cardinal $\leq \kappa$. Let $\alpha > \kappa$ be an ordinal of cofinality $\omega$. Fix a function $F_\alpha$ with domain $\omega$ that is cofinal in $\alpha$. Let $\psi$ be defined similarly to above, but this time using formulas over the language $\{\in, F\}$ where $F$ is a binary relation symbol. Then $(V_\alpha, \in, F_\alpha)$ is a model of the conjunction of the sentences $\psi$, Magidor's $\Phi$ and a formula $\chi$ saying $``F$ is a function with domain $\omega$ that is cofinal in the ordinals". Since $|V_\alpha| \geq \kappa$, the sentence $\psi \wedge \Phi \wedge \chi$ has a model $M$ of size $> |V_\alpha|$.  Because of the usage of $\Phi$, $M$ is without loss of generality of the form $(V_{\beta_\alpha}, \in, F_{\beta_\alpha})$ for some ordinal $\beta_\alpha > \alpha$. As $M \models \psi$, there is an elementary embedding $j: (V_\alpha, \in, F_\alpha) \rightarrow (V_{\beta_\alpha}, \in, F_{\beta_\alpha})$. And because $F_{\beta_\alpha}$ is a function with domain $\omega$ which is cofinal in $\beta_\alpha$ by $\chi$, $j$ has a critical point $\kappa_\alpha$. This shows that the function $\alpha \mapsto \kappa_\alpha$ is regressive on the stationary class $\{\alpha \in \Ord \colon \cof(\alpha) = \omega \}.$ There is a version of Fodor's Lemma showing that every regressive function on a stationary class is constant on an unbounded class (cf. \cite[The very weak class Fodor principle]{git2019}). Thus, $\alpha \mapsto \kappa_\alpha$ is constant on an unbounded class, say with value $\delta$. Then $\delta$ is extendible. We can therefore take the smallest extendible cardinal $\eta$. Now if $\eta \leq \kappa$ we are done. Suppose $\eta > \kappa$. It is known that extendible cardinals are in $C^{(3)}$  (cf.\ \cite[Proposition 23.10]{kan}), hence $\eta \in C^{(3)}$. It follows that $V_\eta  \models ``\kappa$ is the Hanf number of $\mathcal L^2({\wedge^\infty, \exists^\infty,\forall^\infty})"$ (note that this is a $\Pi_3$ statement). Thus we can repeat our argument to get, in $V_\eta$, an extendible cardinal $\nu < \eta$. Being extendible is also $\Pi_3$, and thus, as $\eta \in C^{(3)}$, it is correct about extendibility. So $\nu$ is really extendible. But this contradicts that $\eta$ is the smallest extendible cardinal. 
			
			\medskip Finally, assume (1) and let $T \subseteq \cL^2(\wedge^\infty, \exists^\infty, \forall^\infty)$ be a ${<} \kappa$-satisfiable theory over some vocabulary $\tau$. For simplicity, let us assume that $\tau$ contains only relation symbols and only a single sort symbol. Let $\beth_\alpha = \alpha > \kappa$ be such that $V_\alpha$ verifies the satisfiability of all $T_0 \in \mathcal P_\kappa T$.  Note that $T \subseteq \cL^2(\wedge^\lambda, \exists^\lambda, \forall^\lambda)$ for some $\lambda < \alpha$. By extendibility find some $\beta$ such that $j: V_\alpha \rightarrow V_\beta$ with $\crit(j) = \kappa$ and $j(\kappa) > \alpha$. Then by elementarity $V_\beta \models ``j(T)$ is a ${<}j(\kappa)$-satisfiable theory" and further $j``T \subseteq j(T)$ and $|j``T| = |T| < \alpha < j(\kappa)$. Thus $V_\beta \models ``j``T$ has a model $\mathcal B".$ Now let $\mathcal B'$ be the $\tau$-structure with universe $B$ and with $R^{\mathcal B'} = j(R)^{\mathcal B}$ for all $R \in \tau$. We will show that $\mathcal B' \models T$. For this it is sufficient to show that $V_\beta \models ``\mathcal B \models j(\varphi)"$ implies $\mathcal B' \models \varphi$ for all $\varphi \in \mathcal L^2(\wedge^\lambda, \exists^\lambda,\forall^\lambda)$. We will show this by proving the following claim:
			
			\begin{claim}For every $\varphi \in \cL^2(\wedge^\lambda, \exists^\lambda, \forall^\lambda)$ and every variable assignment $f: j(S) \rightarrow B \cup \bigcup_{n \in \omega} \mathcal P(B^n)$ where $S$ is the set of variables in $\varphi$:
				\[
				V_\beta \models ``\mathcal B \models j(\varphi)[f]" \text{ implies } \mathcal B' \models \varphi[f \circ j].
				\]
			\end{claim}
			Note that $j\upharpoonright S$ is a map $S \rightarrow j(S)$, so $f \circ j$ is a sensible assignment on $S$, and that because $V_\beta$ is a rank initial segment, any assignment $f: j(S) \rightarrow B \cup \bigcup_{n \in \omega} \mathcal P(B^n)$ belongs to it. As every $\varphi \in \cL^2(\wedge^\lambda, \exists^\lambda, \forall^\lambda)$ is equivalent to some $\psi \in \cL^2(\wedge^\lambda, \exists^\lambda, \forall^\lambda)$ in negation normal form, we may restrict attention to such $\psi$.
			
			\medskip For the base case, if $\varphi = R(x_1, \dots, x_n)$ for some $R \in \tau$ and variables $x_1, \dots, x_n \in S$, then $j(\varphi) = j(R)(j(x_1), \dots, j(x_n))$. We assume that \[V_\beta \models ``\mathcal M \models j(R)(f(j(x_1)), \dots, f(j(x_n)))",\] which by definition means $(f(j(x_1)), \dots, f(j(x_n))) \in j(R)^{\mathcal B}$. But then, by definition of $R^{\mathcal B'}$, we have $(f(j(x_1)), \dots, f(j(x_n))) \in R^{\mathcal B'}$, which, again by definition, means $\mathcal B' \models R(x_1, \dots, x_m)[f \circ j]$. The case of negated atomic formulas goes similar. The case of finite applications of $\vee$ is trivial. 
			
			\medskip Consider $\varphi = \bigwedge_{i < \gamma} \chi_i$ and assume $V_\beta \models ``\mathcal B \models j(\varphi)[f]"$. We have that $j(\varphi) = \bigwedge_{i < j(\gamma)} \chi_i^*$ for some formulas $\chi_i^*$. Further, for any $i < \gamma$, $j(\chi_i) = \chi_{j(i)}^*$. We get that $V_\beta \models ``\mathcal B \models j(\chi_i)[f]"$ and therefore by induction hypothesis $\mathcal B' \models \chi_i[f \circ j]$, for any of the $\chi_i$. Thus, $\mathcal B' \models \varphi[f \circ j]$. 
			
			\medskip If $\varphi = \exists W \chi$, then $j(\varphi) = \exists j(W) j(\chi)$. Because $V_\beta \models ``\mathcal B \models \exists j(W)j(\chi)[f]"$, there is a $j(W)$-variant $g$ of $f$ such that $V_\beta \models ``\mathcal B \models j(\chi)[g]"$. By induction hypothesis, $\mathcal B' \models \chi[g \circ j]$. We claim that $g \circ j$ is a $W$-variant of $f \circ j$, and so $\mathcal B' \models \exists W \chi[f \circ j]$. For $x \in S \setminus W$, we have $j(x) \in j(S) \setminus j(W)$. Since $g$ is a $j(W)$-variant of $f$, thus $g(j(x)) = f(j(x))$. 
			
			\medskip Finally, if $\varphi = \forall W \chi$, then $j(\varphi) = \forall j(W) j(\chi)$. We assume that \[V_\beta \models ``\mathcal B \models \forall j(W) j(\chi)[f]",\] and want to show that $\mathcal B' \models \forall W \chi[f \circ j]$. So let $g$ be a $W$-variant of $f \circ j$. Define $h: j(S) \rightarrow B$ by
			\[
			h(y) = \begin{cases}
			g(x), &\text{ if } y \in j``S \text{ and } j(x) = y\\
			f(y), &\text{ otherwise.}
			\end{cases}
			\]
			We claim that $h$ is a $j(W)$-variant of $f$. Let $y \in j(S) \setminus j(W)$. If $y \notin j``S$, then $h(y) = f(y)$ by definition. And if $y \in j``S$, let $j(x) = y$. Then $h(y) = g(x)$. Because $y \notin j(W)$, $x \notin W$, so $h(y) = g(x) = f \circ j(x) = f(y)$ where the middle equality holds because $g$ is a $W$-variant of $f \circ j$. So in any case, $h(y) = f(y)$ and $h$ is thus an $j(W)$-variant of $f$. Thus, $V_\beta \models ``\mathcal B \models j(\chi)[h]"$. Hence, by induction hypothesis, $\mathcal B' \models \chi[h \circ j]$. But $h \circ j = g$. Therefore $\mathcal B' \models \chi[g]$. 
		\end{proof}
	\end{theorem}
	
	By a similar proof to that of Theorem \ref{thm:classSOL}, we can also give a characterisation of extendibility. The only things in need to be changed are that usage of $\cL_{\kappa \kappa}$ can force critical points to be precisely $\kappa$ (as $\cL_{\kappa \kappa}$ can define all ordinals $< \kappa$) and that embeddings with critical point $\kappa$ fix formulas of $\cL_{\kappa \kappa}$ up to renaming. 
	\begin{theorem}
		The following are equivalent for a cardinal $\kappa$.
		\begin{enumerate}
			\item[(1)] $\kappa$ is extendible.
			\item[(2)] $\kappa$ is the compactness number of $\cL^2_{\kappa \kappa}$.
			\item[(3)] $\kappa$ is the compactness number of $\cL^2(\wedge^\infty, \exists^\infty, \forall^\infty, \vee^\kappa)$.
			\item[(4)] $\kappa$ is the weak compactness number of $\cL^2(\wedge^\infty, \exists^\infty, \forall^\infty, \vee^\kappa)$.
			\item[(5)] $\kappa$ is the Hanf number of $\cL^2(\wedge^\infty, \exists^\infty, \forall^\infty, \vee^\kappa)$.
			\item[(6)] $\kappa$ is the ULST number of $\cL^2(\wedge^\infty, \exists^\infty, \forall^\infty, \vee^\kappa)$.
			\item[(7)] $\kappa$ is the SULST number of $\cL^2(\wedge^\infty, \exists^\infty, \forall^\infty, \vee^\kappa)$.
		\end{enumerate}
	\end{theorem}
	
	Similarly to how the compactness, and SULST numbers of $\cL^2$ give us an extendible cardinal, it is known that a cardinal $\kappa$ is $C^{(n)}$-extendible iff $\kappa = \comp(\sln_{\kappa \omega})$ iff $\kappa = \SULS(\sln_{\kappa \omega})$ (cf. \cite[Theorem 4.9]{bon2020} for the compactness case; the $\SULS$ case can be obtained by a slight addition to arguments in either \cite{gitman2025upward} or \cite{boney2025} which cover the (S)ULST number of finitary $\cL^2$). 
	An analogous proof to the last one, which uses $\Phi^{(n)}$ from Fact \ref{factSL} to express that some $\alpha$ is in $C^{(n)}$, yields:
	\begin{theorem}
		The following are equivalent for a cardinal $\kappa$:
		\begin{enumerate}
			\item[(1)] $\kappa$ is $C^{(n)}$-extendible.
			\item[(2)] $\kappa$ is the compactness number of $\sln_{\kappa \omega}$.
			\item[(3)] $\kappa$ is the compactness number of $\sln_{\kappa \omega}(\wedge^\infty, \exists^\infty,\forall^\infty)$.
			\item[(4)] $\kappa$ is the weak compactness number of $\sln_{\kappa \omega}(\wedge^\infty, \exists^\infty,\forall^\infty)$.
			\item[(5)] $\kappa$ is the Hanf number of $\sln_{\kappa \omega}(\wedge^\infty, \exists^\infty,\forall^\infty)$.
			\item[(6)] $\kappa$ is the ULST number of $\sln_{\kappa \omega}(\wedge^\infty, \exists^\infty,\forall^\infty)$.
			\item[(7)] $\kappa$ is the SULST number of $\sln_{\kappa \omega}(\wedge^\infty, \exists^\infty,\forall^\infty)$.
		\end{enumerate}
	\end{theorem}
	This result allows us to characterize $\VP$ in the following way. Note that (1) strengthens Makowsky's result mentioned in Fact \ref{fact:VP}. That the cases of (4) and (5) restricted to set-sized logics are equivalent to $\VP$ was also known (cf.\cite{gitman2025upward} and \cite{boney2025}), so (4) and (5) can also be seen as strengthening these results. 
	Also here, note that (2) and (3) are much stronger than the respective schemas restricted to set-sized logics.
	\begin{corollary}
		$\VP$ is equivalent to the following axiom schemas.
		\begin{enumerate}
			\item[(1)] Every (possibly class-sized) logic $\cL$ with $\cL \leq \sln_{\kappa \omega}(\wedge^\infty, \exists^\infty, \forall^\infty)$ for some $n$ and some $\kappa$ has a compactness number.
			\item[(2)] Every (possibly class-sized)  logic $\cL$ with $\cL \leq \sln_{\kappa \omega}(\wedge^\infty, \exists^\infty, \forall^\infty)$ for some $n$ and some $\kappa$  has a weak compactness number.
			\item[(3)] Every (possibly class-sized)  logic $\cL$ with $\cL \leq \sln_{\kappa \omega}(\wedge^\infty, \exists^\infty, \forall^\infty)$ for some $n$ and some $\kappa$  has a Hanf number.
			\item[(4)] Every (possibly class-sized)  logic $\cL$ with $\cL \leq \sln_{\kappa \omega}(\wedge^\infty, \exists^\infty, \forall^\infty)$ for some $n$ and some $\kappa$  has a ULST number.
			\item[(5)] Every (possibly class-sized)  logic $\cL$ with $\cL \leq \sln_{\kappa \omega}(\wedge^\infty, \exists^\infty, \forall^\infty)$ for some $n$ and some $\kappa$  has an SULST number.
		\end{enumerate}
	\end{corollary}
	\section{Weak Vopěnka's Principle and L\"o\-wen\-heim-Sko\-lem properties}\label{sec:ClassWVP}
	Both $\VP$ and $\WVP$ are characterized in terms of the existence of large cardinals, namely of $C^{(n)}$-extendible and of $\Pi_n$-strong cardinals, respectively. Moreover, with increasing $n$, the $C^{(n)}$-extendible and $\Pi_n$-strong cardinals form respective hierarchies of increasing consistency strength below $\VP$ and $\WVP$. There is thus a strong analogy of $\VP$ and $\WVP$ in terms of their characterizations by large cardinals. $\VP$ is also equivalent to the existence of compactness and of LST numbers for every set-sized logic. This begs the question whether $\WVP$ also has a characterization in terms of model-theoretic properties of logics. We will show that this is indeed the case. More concretely, we will show that certain downwards Löwenheim-Skolem properties for the class version $\sln(\vee^\infty, \forall^\infty)$ of sort logic are equivalent to the existence of $\Pi_n$-strong cardinals. As $\WVP$ is equivalent to the existence of $\Pi_n$-strong cardinals for every $n$, this results in a characterization of $\WVP$.
	
	Boney and Osinski \cite{boney2025} and Holy, Lücke, and Müller \cite{holy2024} also give independent model-theoretic characterizations of $\WVP$. Both approaches use compactness properties though. As $\VP$ is characterized by both compactness and downwards Löwenheim-Skolem principles, our analysis pushes the model-theoretic analogy of $\VP$ and $\WVP$ further by providing an equivalence to the latter type of properties. The analogy is also pushed in a different sense, as, next to giving an equivalence to the global WVP, we characterize the hierarchy of $\Pi_n$-strong cardinals below $\WVP$, similarly to how the $C^{(n)}$-extendible cardinals can be characterized model-theoretically (cf. Section \ref{subsec:SOLSL}).

	\subsection{Statement of the main results}\label{subsec:WVPLS}
	Recall that if $\LST(\cL) = \kappa$ for some logic $\cL$, we demand that for any $\varphi \in \cL$ and any $\mathcal A \models \varphi$, there is some substructure $\mathcal B \models \varphi$ such that $|B| < \kappa$, \emph{provided} $\mathcal A$ is a $\tau$-structure for some vocabulary $\tau$ of size $|\tau|<\kappa$. For the LS number, we usually do not need such a restriction in the size of vocabularies, as in usual definitions of set-sized strong logics -- for instance, the one we employed -- the amount of non-logical symbols that a sentence can use is bounded anyway. But for class logics this is not the case. 
	
	We therefore consider the following version of the LS number.
	\begin{definition}
		Let $\mathcal L$ be a logic and $\lambda$ a cardinal. A cardinal $\kappa$ is an  \emph{$\text{LS}^\lambda$ number of $\mathcal L$} iff any satisfiable sentence of $\cL$ over a vocabulary of size $< \lambda$ has a model of size $< \kappa$. Should such a cardinal exist, we denote the smallest such by $\text{LS}^\lambda(\mathcal L)$. 
	\end{definition}
	We first mention the following unpublished result by Wilson. We omit the proof, as it can be carried out with similar arguments as that of Theorem \ref{thm:LS-Pin-strong}.
	\begin{theorem}[Wilson]
		The following are equivalent for a cardinal $\kappa$:
		\begin{enumerate}
			\item[(1)] $\kappa$ is the smallest strong cardinal.
			\item[(2)] $\kappa = \text{LS}^\omega(\cL^2(\vee^\infty, \forall^\infty))$. 
			\item[(3)] $\kappa$ is the smallest cardinal such that $\kappa = \text{LS}^\kappa(\cL^2(\vee^\infty, \forall^\infty))$.
		\end{enumerate}
	\end{theorem}
	We will show the following theorem that connects LS numbers of class logics of the form $\sln(\vee^\infty, \forall^\infty)$ to local forms of ``Ord is Woodin".
	\begin{theorem}\label{thm:LS-Pin-strong}
		The following are equivalent for every natural number $n \geq 1$ and cardinal $\kappa$:
		\begin{enumerate}
			\item[(1)] $\kappa$ is the smallest $\Pi_n$-strong cardinal.
			\item[(2)] $\kappa$ is the smallest $C^{(n)}$-strong cardinal, i.e., the smallest $A$-strong cardinal with $A = C^{(n)}$.\footnote{Note that the term \emph{$C^{(n)}$-strong cardinal} is also used to refer to a strong cardinal $\kappa$ that has the property that the embeddings $j$ witnessing strength can be demanded to fulfil $j(\kappa) \in C^{(n)}$. It is known that being $C^{(n)}$-strong in this sense is equivalent to merely being strong (cf. \cite[Proposition 1.2]{bag2012}). In the following we will always mean that $\kappa$ is $A$-strong with $A = C^{(n)}$ when speaking of a \emph{$C^{(n)}$-strong cardinal}.}
			\item[(3)] $\kappa = \text{LS}^\omega(\sln(\vee^\infty, \forall^\infty))$. 
			\item[(4)] $\kappa$ is the smallest cardinal such that $\kappa$ is an $\text{LS}^\kappa$ number of $\sln(\vee^\infty, \forall^\infty)$.
			\item[(5)] $\kappa$ is the smallest cardinal such that $\kappa = \text{LS}^\kappa(\sln(\vee^\infty, \forall^\infty))$.
		\end{enumerate}
		\begin{proof}
			Cf. Proof \ref{LS-Pin-strong}.
		\end{proof}
	\end{theorem}
	
	In particular, by Fact \ref{fact:PinWVP}, the above implies that we get a local equivalence of the existence of $\text{LS}^\omega(\sln(\vee^\infty, \forall^\infty))$ to fragments of $\WVP$.
	\begin{corollary}
		The following are equivalent:
		\begin{enumerate}
			\item[(1)] $\WVP(\Pi_n)$.
			\item[(2)] $\sln(\vee^\infty, \forall^\infty)$ has an $\text{LS}^\omega$ number.
		\end{enumerate}
	\end{corollary}
	As for the case of compactness numbers of $\cL^2(\wedge^\infty, \exists^\infty, \forall^\infty)$, by minor adaptations to the proof of Theorem \ref{thm:LS-Pin-strong}, we obtain the following result. 
	\begin{theorem}\label{thm:LS-Pin-strong-exact}
		The following are equivalent for a cardinal $\kappa$:
		\begin{enumerate}
			\item[(1)] $\kappa$ is $\Pi_n$-strong.
			\item[(2)] $\kappa = \text{LS}^\kappa(\sln_{\kappa \omega}(\vee^\infty, \forall^\infty))$. 
		\end{enumerate}
	\end{theorem}
	Combining this with Fact \ref{fact:BagariaWilson}, we get a characterisation of $\WVP$ in terms of Löwenheim-Skolem properties.
	\begin{corollary}\label{cor:WVPLS}
		The following are equivalent.
		\begin{enumerate}
			\item[(1)] WVP.
			\item[(2)] For every $n$, $\sln(\vee^\infty, \forall^\infty)$ has an $\text{LS}^{\omega}$ number.
			\item[(3)] Every (possibly class-sized) logic $\cL$ with $\cL \leq \sln_{\kappa \omega}(\vee^\infty, \forall^\infty)$ for some $n$ and some $\kappa$ has an $\text{LS}^\omega$ number.
			\item[(4)] Every (possibly class-sized) logic $\cL$ with $\cL \leq \sln_{\kappa \omega}(\vee^\infty, \forall^\infty)$ for some $n$ and some $\kappa$ has an $\text{LS}^\lambda$ number for every $\lambda$.  
		\end{enumerate}
		\begin{proof}
			Clearly (4) implies (3) and (3) implies (2). By Fact \ref{fact:BagariaWilson}, $\WVP$ is equivalent to the existence of a $\Pi_n$-strong cardinal for every $n$, and of a proper class of $\Pi_n$-strong cardinals for every $n$. In particular, (2) implies (1) by Theorem \ref{thm:LS-Pin-strong}. So left to show is that (1) implies (4). Assume (1) and take any cardinal $\lambda$ and a logic $\cL \leq \sln_{\kappa \omega}(\vee^\infty, \forall^\infty)$ for some $\kappa$. By $\WVP$, we get a $\Pi_n$-strong cardinal $\gamma > \kappa + \lambda$. By Theorem \ref{thm:LS-Pin-strong-exact}, $\gamma = \text{LS}^\gamma(\sln_{\gamma \omega}(\vee^\infty, \forall^\infty))$. Then $\text{LS}^\lambda(\cL) \leq \text{LS}^\lambda(\sln_{\kappa \omega}(\vee^\infty, \forall^\infty)) \leq \text{LS}^\gamma (\sln_{\gamma \omega}(\vee^\infty, \forall^\infty))$. 
		\end{proof}
	\end{corollary}

	\subsection{Proof of the main theorem}\label{sec:ProofLS-Pin-strong}
	In this section, we give a proof of Theorem \ref{thm:LS-Pin-strong}. We will use the framework of \emph{$\mathscr{P}$-structures}, which we will present first. It constitutes an alternative approach to extenders and similarly allows to approximate elementary embeddings of the universe. 
	
	\subsubsection{$\mathscr{P}$-structures as an alternative to extenders}
	The technicalities of the framework presented below are due to Wilson \cite{wilson2022large}, and similar constructions can already be found in \cite{neeman2004mitchell} and \cite{zeman2002}. Instead of sequences of ultrafilters as for extenders, this approach constructs elementary embeddings from homomorphism of certain Boolean algebras called \emph{$\mathscr P$-structures}. For a set $X$, we write $X^{< \omega}$ for the set \[\{(a_1, \dots, a_k) \colon k \in \omega \text{ and } a_1, \dots, a_k \in X\}\] of finite sequences of members of $X$. We write $\subsetneq$ for the proper initial segment relation, i.e., $(a_1, \dots, a_k) \subsetneq (b_1, \dots, b_l)$ iff $k < l$ and $a_i =  b_i$ for all $i \leq k$. We further write $\supsetneq$ for the reverse of the relation $\subsetneq$.
	\begin{definition}[Wilson {\cite{wilson2022large}}]\label{def:p-str}
		Let $X$ be a transitive set. A \emph{$\mathscr P$-structure} is a structure
		\[
		\mathscr P_X = (\mathcal P(X^{< \omega}), \cap, \setminus, X^k, \text{WF},\pi^{-1}_{k,(i_1, \dots, i_j)}, \text{BP}_k)_{j,k < \omega, 1 \leq i_1, \dots, i_j \leq k}.
		\]
		such that 
		\begin{enumerate}
			\item[(1)] $\cap$ is intersection, interpreted as a binary operation.
			\item[(2)] $\setminus$ is complementation, interpreted as a unary operation.
			\item[(3)] $X^k \in \mathcal P(X^{< \omega})$ is a constant.
			\item[(4)] $\text{WF}$ is a unary relation such that $A \in \text{WF}$ iff $A \subseteq P(X^{< \omega})$ and $(A,\supsetneq)$ is well-founded.
			\item[(5)] $\pi^{-1}_{k,(i_1, \dots, i_j)}$ is a function $\mathcal P(X^j) \rightarrow \mathcal P(X^k)$ defined by
			\[
			\pi^{-1}_{k,(i_1, \dots, i_j)}(A) = \{(a_1, \dots, a_k) \in X^k \colon (x_{i_1}, \dots, x_{i_j}) \in A \}.
			\]
			\item[(6)] $\text{BP}_k$ is a function $\mathcal P(X^{k+1}) \rightarrow \mathcal P(X^{k+1})$ defined by
			\[
			\text{BP}_k(A) = \{(a_1, \dots, a_{k+1}) \in X^k \colon \exists z \in x_{k+1}((x_1, \dots, x_k,z) \in A)\}.
			\]
		\end{enumerate}
		If further the structure is of the following form, for an additional constant $c^{< \omega}$ where $c \subseteq X$, we call $\mathcal P_{X,c}$ a \emph{pointed $\mathscr P$-structure}.
		\begin{enumerate}
		\item[(7)] $\mathscr P_{X,c} = (\mathcal P(X^{< \omega}), \cap, \setminus, X^k, \text{WF},\pi^{-1}_{k,(i_1, \dots, i_j)}, \text{BP}_k, c^{< \omega})_{j,k < \omega, 1 \leq i_1, \dots, i_j \leq k}$.
		\end{enumerate}
	\end{definition}
	A \emph{homomorphism of $\mathscr P$-structures} $h: \mathscr P_X \rightarrow \mathscr P_Y$ is simply a homomorphism in the usual model-theoretic sense, i.e., a map preserving the relations, constants and functions defined on $\mathscr P_X$. For instance, $h(A \cap B) = h(A) \cap h(B)$, or $A \in \text{WF}^{\mathscr P_X}$ implies $A \in \text{WF}^{\mathscr P_Y}$. A \emph{homomorphism of pointed $\mathscr P$-structures} $h: \mathscr P_{X,c} \rightarrow \mathscr P_{Y,d}$ additionally preserves the constant $c^{< \omega}$, i.e., $h(c^{< \omega}) = d^{< \omega}$. 
	
	The way in which we will use $\mathscr P$-structures is by the following result, which shows how they give rise to elementary embeddings.
	\begin{fact}[Wilson {\cite{wilson2022large}}]\label{fact:p-str}
		Let $X$ and $Y$ be transitive, $c\subseteq X$, $d \subseteq Y$, and $h: \mathscr P_{X,c} \rightarrow \mathscr P_{Y,d}$ be a homomorphism of pointed $\mathscr P$-structures. Then there is a transitive class $M$ and an elementary embedding $j: V \rightarrow M$ such that $Y \subseteq j(X)$, $h(A) = j(A) \cap Y^{< \omega}$ for all $A \subseteq X^{< \omega}$, and $j(c) \cap Y =  d$. 
	\end{fact}
	Let us further note that there is a trivial homomorphism $\mathscr P_X \rightarrow \mathscr P_Y$ if $Y \subseteq X$.
	\begin{fact}[Wilson {\cite{wilson2022large}}]\label{fact:trivialHom}
		Let $X$ and $Y$ be transitive such that $Y \subseteq X$. Then $h: \mathcal P(X^{< \omega}) \rightarrow \mathcal P(Y^{< \omega})$ defined by $A \mapsto A \cap Y^{< \omega}$ is a homomorphism of $\mathscr P$-structures $h: \mathscr P_X \rightarrow \mathscr P_Y$. 
	\end{fact}
	
	\subsubsection{Main proof}
	Before we give the main proof, let us mention that, as Wilson \cite{wilson2022large} and Bagaria and Wilson \cite{bagaria2023weak} observe, in the definition of $\Pi_n$-strong and of $A$-strong cardinals, we may equivalently drop the assumption that $j(\kappa) > \lambda$. In the case of $A$-strong cardinals, this then additionally requires to exchange the condition $j(A \cap V_\kappa) \cap V_\lambda = A \cap V_\lambda$ for $j(A \cap V_\lambda) \cap V_\lambda = A \cap V_\lambda$. We will use these observations tacitly below.
	
	\begin{Proof}[Proof of Theorem \ref{thm:LS-Pin-strong}]\label{LS-Pin-strong}
		That the first two items are equivalent is shown (but not explicitly stated) in the proof of \cite[Proposition 5.9]{bagaria2023weak}. We will show the equivalence of (2), (3), (4), and (5). For this, it suffices to show that (3) implies (2), and that (2) implies (4). This holds, because if (4) is true, and $\kappa$ is the smallest cardinal such that $\kappa$ is an $\text{LS}^\kappa$ number of $\sln(\vee^\infty, \forall^\infty)$, then in particular, $\gamma = \text{LS}^\omega(\sln(\vee^\infty, \forall^\infty))$ exists and $\gamma \leq \text{LS}^\kappa(\sln(\vee^\infty, \forall^\infty)) \leq \kappa$. But if (3) implies (2), then this means that $\gamma$ is the smallest $C^{(n)}$-strong cardinal. And if (2) implies (4), then $\gamma$ is the smallest cardinal such that $\gamma$ is an $\text{LS}^\gamma$ number of $\sln(\vee^\infty, \forall^\infty)$. But this implies $\kappa = \gamma$ and so (4) implies (3). This shows the equivalence of (2), (3), and (4). Finally, notice that the above calculation also shows that $\kappa = \text{LS}^\kappa(\sln(\vee^\infty, \forall^\infty))$, and so (5) is equivalent to the other statements. 
		
		\medskip
		
		So let us first show that (3) implies (2). The following two claims suffice for this.
		\begin{claim}\label{ThisClaim}
			If $\kappa=\text{LS}^\omega(\sln(\vee^\infty, \forall^\infty))$, then for all $\lambda > \kappa$ which are limits of $C^{(n)}$, there is a \emph{$\lambda$-$C^{(n)}$-strong cardinal} $\leq \kappa$, i.e., some $\delta \leq \kappa$ such that there is $j: V \to M$ with $\crit(j) = \delta$, $V_\lambda \subseteq M$, and $j(C^{(n)} \cap V_\lambda) \cap V_\lambda = C^{(n)} \cap V_\lambda$.
		\end{claim}
	
		\begin{claim}\label{ThatClaim}
			If $\kappa$ is $C^{(n)}$-strong, then $\kappa$ is an $\text{LS}^\kappa(\sln(\vee^\infty, \forall^\infty))$ number.
		\end{claim}
		If we showed both claims, then (3) implies (2) in the following way: Assume (3), so that $\kappa = \text{LS}^{\omega}(\sln(\vee^\infty, \forall^\infty))$. As there are set many cardinals $\leq \kappa$ but a proper class of $\lambda > \kappa$, Claim \ref{ThisClaim} implies that there has to be a cardinal $\delta \leq \kappa$ that is $\lambda$-$C^{(n)}$-strong for arbitrarily large $\lambda$ and hence $C^{(n)}$-strong. Suppose without loss of generality that $\delta$ is the smallest $C^{(n)}$-strong cardinal. Then to show (2), we have to show $\delta = \kappa$, which follows from $\kappa \leq \delta$.  Claim \ref{ThatClaim} implies that $\delta$ is an $\text{LS}^\delta$ number of $\sln(\vee^\infty, \forall^\infty)$. But then clearly $\kappa = \text{LS}^{\omega}(\sln(\vee^\infty, \forall^\infty)) \leq \text{LS}^{\delta}(\sln(\vee^\infty, \forall^\infty)) \leq \delta$. 
		
		\medskip To prove Claim 31, we let $\lambda > \kappa$ be a limit of $C^{(n)}$, and assume that $\kappa = \text{LS}^\omega(\sln(\vee^\infty, \forall^\infty))$. We want to show that some cardinal $\leq \kappa$ is $\lambda$-$C^{(n)}$-strong. For this, it is sufficient to show that there is a homorophism $h: \mathscr P_{V_\kappa, C^{(n)} \cap V_\kappa} \rightarrow \mathscr P_{V_\lambda, C^{(n)} \cap V_\lambda}$ of pointed $\mathscr P$-structures. For then, Wilson's Fact \ref{fact:p-str} implies that there is $j: V \rightarrow M$ such that $V_\lambda \subseteq M$ and $j(C^{(n)} \cap V_\kappa) \cap V_\lambda = C^{(n)} \cap V_\lambda$. In particular, as $\lambda$ is a limit of $C^{(n)}$, there is some $\alpha > \kappa$ such that $\alpha \in C^{(n)}$ and $\alpha \in j(C^{(n)} \cap V_\kappa) \cap V_\lambda$. This implies that $j(\kappa) > \lambda$ and so $j$ has a critical point $\crit(j) \leq \kappa$. Then $j$ witnesses that $\crit(j)$ is $\lambda$-$C^{(n)}$-strong. So we showed that it is sufficient to find a homomorphism  $h: \mathscr P_{V_\kappa, C^{(n)} \cap V_\kappa} \rightarrow \mathscr P_{V_\lambda, C^{(n)} \cap V_\lambda}$ of pointed $\mathscr P$-structures. To obtain this, it is sufficient to show the following claim. 
		
		\begin{claim}\label{cl:nohom} 
			There is a sentence $\varphi \in \sln(\vee^\infty, \forall^\infty)$ such that for any limit ordinal $\alpha$:
			\[
			(V_\alpha, \in) \models \varphi \text{ iff there is no homomorphism } h: \mathscr P_{V_\kappa, C^{(n)} \cap V_\kappa} \rightarrow \mathscr P_{V_\alpha, C^{(n)} \cap V_\alpha}.
			\]
		\end{claim}
		Let us argue that showing Claim \ref{cl:nohom} is sufficient to derive the existence of a homomorphism  $h: \mathscr P_{V_\kappa, C^{(n)} \cap V_\kappa} \rightarrow \mathscr P_{V_\lambda, C^{(n)} \cap V_\lambda}$. Suppose there is no such homomorphism. Let $\varphi$ be the sentence from Claim \ref{cl:nohom} and let $\varphi^*$ be the conjunction of $\varphi$ and Magidor's $\Phi$ (cf.\ Fact \ref{fact:MagPhi}). Then $V_\lambda \models \varphi^*$. Therefore, by $\kappa$ being $\text{LS}^\omega(\sln(\vee^\infty, \forall^\infty))$ there is a structure $(M, E^M)$ of size $< \kappa$ with $(M, E^M) \models \varphi^*$. As $(M, E^M) \models \Phi$, we can without loss of generality assume that $(M, E^M) = (V_\alpha, \in)$ for some $\alpha < \kappa$, for which in particular $V_\alpha \subseteq V_\kappa$. But by this fact and Fact \ref{fact:trivialHom}, the map $A \mapsto A \cap V_\alpha^{< \omega}$ is a (trivial) homomorphism $h: \mathscr P_{V_\kappa} \rightarrow \mathscr P_{V_{\alpha}}$ of $\mathscr P$-structures. Because $h((C^{(n)} \cap V_\kappa)^{< \omega}) = (C^{(n)} \cap V_\kappa)^{< \omega} \cap V_\alpha^{< \omega} = (C^{(n)} \cap V_\alpha)^{< \omega}$, this is also a homomorphism of pointed $\mathscr P$-structures $\mathscr P_{V_\kappa, C^{(n)} \cap V_\kappa} \rightarrow \mathscr P_{V_\alpha, C^{(n)} \cap V_\alpha}$. Thus $V_\alpha \not \models \varphi$. Contradiction.
		
		\medskip So we reduced our aim to proving Claim \ref{cl:nohom}. For this, we may instead show the following assertion.
		\begin{claim}\label{Claim:hom}
			There is a sentence $\psi \in \sln(\wedge^\infty, \exists^\infty)$ such that for any limit ordinal $\alpha$:
			\[
			(V_\alpha, \in) \models \psi \text{ iff there is a homomorphism } \mathscr P_{V_\kappa, C^{(n)} \cap V_\kappa} \rightarrow \mathscr P_{V_\alpha, C^{(n)} \cap V_\alpha}
			\]
		\end{claim}  
		Given Claim \ref{Claim:hom}, taking $\varphi = \neg \psi$ (and pushing negations through the infinitary quantifiers and conjunctions) proves Claim \ref{cl:nohom}. To show Claim \ref{Claim:hom}, we let \[\psi \mathrel = \exists(X_A \colon A \in \mathcal P(V_\kappa^{<\omega}))(\bigwedge_{i = 1}^7 \psi_i),\] where each $X_A$ is a unary second-order variable and each $\psi_i$ will be specified below. The purpose of the sentences is the following. If for some $V_\alpha$ there is a sequence $(X_A^{V_\alpha} \colon A \in \mathcal P(V_\kappa^{< \omega}))$ witnessing that $V_\alpha$ satisfies $\psi_i$, then the map $h: \mathcal P(V_\kappa^{< \omega}) \rightarrow \mathcal P(V_\alpha^{< \omega})$ defined by $A \mapsto X_A^{V_\alpha}$ shall preserve the clause of Definition \ref{def:p-str} corresponding to $\psi_i$. Then if $V_\alpha \models \psi$, this map is a full homomorphism. Let us go through the conjuncts of $\psi$, each time arguing why the map $h$ is a homomorphism with respect to the intended part of the structure.
		\begin{equation*}
		\begin{split} 
		& \psi_1 = \bigwedge_{\substack {A,B,C \in \mathcal P(V_\kappa^{< \omega}) \\ A \cap B = C}} \forall x ((X_A(x) \wedge X_B(x)) \leftrightarrow X_C(x)). \\
		& \psi_2 = \bigwedge_{\substack {A,B \in \mathcal P(V_\kappa^{< \omega}) \\ V_\kappa^{< \omega}\setminus A = B} } \forall x(X_{V_\kappa^{< \omega}}(x) \wedge \neg X_A(x) \leftrightarrow X_B(x)).\\
		\end{split}
		\end{equation*}
		Clearly, $\psi_1$ codes that $X_A^{V_\alpha} \cap X_B^{V_\alpha} = X_C^{V_\alpha}$. Because $V_\kappa^{< \omega} \setminus V_\kappa^{< \omega} = \emptyset$, by $\psi_2$ we get that $h(\emptyset) = \emptyset$ and $h(V_\kappa^{< \omega}) = V_\kappa^{< \omega}$. Using this, satisfaction of $\psi_2$ implies preservation of complements.
		\begin{equation*}
		\psi_3 = \bigwedge_{k \in \omega} \forall x (``x \text{ is a sequence of length } k" \leftrightarrow X_{V_\kappa^k}(x)).
		\end{equation*}
		Here we take that $x$ is a sequence of length $k$ simply written out as a set-theoretic statement in first-order logic. As $V_\alpha$ knows of all sequences of length $k$, this implies that $h(V_\kappa^k) = V_\alpha^k$. The sentence $\psi_4$ is a conjunction
		\begin{equation*}
		\psi_4 = \bigwedge_{\substack{A \subseteq X^{< \omega} \\ (A,\supsetneq) \text{ well-founded}}} \chi^A_1,
		\end{equation*}
		where each $\chi_1^A$ is given by
		\[
		\chi_1^A = \neg \exists F(\dom(F) = \omega \wedge \forall n \in \omega(X_A(F(n)))\wedge \forall n \in \omega (F(n) \subsetneq F(n+1)) ).
		\]
		Note that $(A,\supsetneq)$ is well-founded iff there is a function $f$ with domain $\omega$ and $\ran(f) \subseteq A$ such that $f(n) \subsetneq f(n+1)$ for all $n \in \omega$. This is coded by the above sentence, using a binary second-order variable $F$ which we implicitly assume to be functional (note that we could express this).
		
		\medskip In the formula below, for some $k \in \omega$ and $1 \leq i_1, \dots, i_j \leq k$, and further $A \in \mathcal P(V_\kappa^{< \omega})$, we write $\varphi_{k,i_1, \dots, i_j}(x)$ for the formula $\exists y_1, \dots, y_k(x = (y_1, \dots, y_k) \wedge X_A((y_{i_1}, \dots, y_{i_j})))$, expressing that $x$ is some $k$-tuple $x = (y_1, \dots, y_k)$ and $(y_{i_1}, \dots, y_{i_j})$ is a member of $X_A$. Then let
		\begin{equation*}
		\begin{split}
		\psi_5 =  \bigwedge_{\substack {A,B \in \mathcal P(V_\kappa^{\omega}), \\ B = \{(x_1, \dots, x_k) \colon (x_{i_1}, \dots, x_{i_j}) \in A) \}, \\  k \in \omega, 1 \leq i_1, \dots, i_j \leq k}} 	\forall x (X_B(x) \leftrightarrow \varphi_{k,i_1, \dots, i_j}^A(x)).
		\end{split}
		\end{equation*}
		Then $\psi_5$ simply codes the defining conditions of $\pi^{-1}_{k,(i_1, \dots, i_j)}(A) = B$. 
		
		\medskip Now write $\chi_k^A$ for the formula 
		\[
		\exists y_1, \dots, y_{k+1}(x = (y_1, \dots, y_{k+1}) \wedge \exists z(E(z,x_{k + 1}) \wedge X_A((y_1, \dots, y_k, z))) ),\] 
		expressing that $x = (y_1, \dots, y_{k+1})$ is some $(k+1)$-tuple and there is some $z \in y_{k+1}$ such that $(y_1, \dots, y_k,z)$ is a member of $X_A$, and let
		\begin{equation*}
		\begin{split}
		& \psi_6 = \bigwedge_{\substack{B = \{(x_1, \dots, x_{k + 1}) \colon \\ \exists z \in x_{k+1} (x_1, \dots, x_k, z) \in A \} \\ A, B \in \mathcal P(V_\kappa^{\omega}), k \in \omega } } \forall x(X_B(x) \leftrightarrow \chi_{k}^A(x)). \\
		\end{split}
		\end{equation*}
		Again, $\psi_6$ simply codes the defining conditions of $\text{BP}_k(A) = B$. Finally, let
		\[
		\psi_7 = \forall x (X_{C^{(n)} \cap V_\kappa}(x) \leftrightarrow (\Phi^{(n)})^{\{y \colon E(y,V_x)\}}).
		\]
		Here $(\Phi^{(n)})^{\{y \colon E(y,V_x)\}}$ is the relativization of $\Phi^{(n)}$ from Fact \ref{factSL} to $V_x$, i.e., to the rank initial segment cut off at some ordinal $x$. Then $(\Phi^{(n)})^{\{y \colon E(y,V_x)\}}$ truthfully codes that $V_x$ is such that $x \in C^{(n)}$. Then $\psi_7$ uses that $V_\alpha$ is correct about $V_\beta$ for $\beta < \alpha$ and the sentence $\Phi^{(n)}$ to express that $X_{C^{(n)} \cap V_\kappa}^{V_\alpha}$ is actually $V_\alpha \cap C^{(n)}$. This ends the construction of $\psi$ and we thus showed what we promised, and hence Claim \ref{ThisClaim}.
		
		\medskip And now we show Claim \ref{ThatClaim}, i.e., if $\kappa$ is $C^{(n)}$-strong, then $\kappa$ is an $\text{LS}^{\kappa}$ number of $\sln(\vee^\infty, \forall^\infty)$. So let $\tau$ be some vocabulary of size $< \kappa$, $\mathcal A$ a $\tau$-structure and $\varphi \in\sln(\vee^\infty, \forall^\infty)[\tau]$ such that $\mathcal A \models \varphi$. Without loss of generality we can assume that $\tau \in H_\kappa$. Our goal is to show that there is a structure of size $< \kappa$ satisfying $\varphi$. For this purpose take a $\lambda > \kappa$ which is a limit point of $C^{(n)}$ with $\mathcal A, \varphi \in V_\lambda$. By $C^{(n)}$-strength of $\kappa$ there is an elementary embedding $j: V \rightarrow M$ with $\crit(j) = \kappa$, $j(\kappa) > \lambda$, $V_\lambda \subseteq M$ and $C^{(n)} \cap V_\lambda = j(C^{(n)} \cap V_\kappa) \cap V_\lambda$. Note that $\mathcal A \in M$ and that $j(\tau) = \tau$. We will show that $M \models ``\mathcal A \models j(\varphi)"$. This makes sense, because $j(\varphi)$ is a $\tau$-sentence, as $j(\tau) = \tau$. Furthermore, this is sufficient, because then $M \models ``|A| < \lambda < j(\kappa) \wedge \mathcal A \models j(\varphi)"$, and so by elementarity of $j$, in $V$ there has to be some model of $\varphi$ of size $< \kappa$. We show the following claim, which will be sufficient. Let $S$ be the set of variables appearing in $\varphi$. Note that the set of variables used in $j(\varphi)$ is $j(S)$.
		\begin{claim}\label{claim:slninduction} 
			For every structure $\mathcal B \in V_\lambda$ and for every subformula $\psi$ of $\varphi$ and any variable assignment $f: j(S) \rightarrow V_\lambda$ with $f \in M$, if $\mathcal B \models \psi[f \circ (j \upharpoonright S)]$, then $M \models ``\mathcal B  \models j(\psi)[f]"$. 
		\end{claim}
		Note that $j\upharpoonright S$ is a map $S \rightarrow j(S)$, so $f \circ (j \upharpoonright S)$ is a sensible assignment on $S$. Furthermore, $\varphi$ is a sentence and so concrete assignments do not play a role in its evaluation. Thus we get that $\mathcal A \models \varphi$ implies that $M \models ``\mathcal A \models j(\varphi)"$ by the claim, by just taking $f$ to be any assignment and $\mathcal B$ to be $\mathcal A$. 
		
		\medskip Before we show Claim \ref{claim:slninduction}, we show the following stronger assertion for formulas of $\sln$, which gives us the base case for our induction for Claim \ref{claim:slninduction}.
		\begin{claim}\label{claim:secondslninduction} 
			For every structure $\mathcal B \in V_\lambda$ and for every subformula $\psi$ of $\varphi$ such that $\psi \in \sln$, and every $f: j(S) \rightarrow V_\lambda$ with $f \in M$, we have $\mathcal B \models \psi[f \circ (j \upharpoonright S)]$ iff $M \models ``\mathcal B \models j(\psi)[f]"$.
		\end{claim}
		To show Claim \ref{claim:secondslninduction}, notice that because $\psi \in \sln$ and $\tau \in H_\kappa$, the only difference between $\psi$ and $j(\psi)$ is a possible renaming of variables, so, for example,  application of $j$ to $\psi = R(x_1, \dots, x_n)$ gives us $j(\psi) = R(j(x_1), \dots, j(x_n))$. This means that if $y$ is some variable appearing in $j(\psi)$ it is of the form $j(x)$ for some $x \in S$. But then $(f \circ (j\upharpoonright S))(x) = f(j(x)) = f(y)$ and so \[\mathcal B \models R(f(j(x_1)), \dots, f(j(x_n))) \text{ iff } M \models ``\mathcal B \models R(f(y), \dots, f(y))".\] These calculations that $j(\psi)$ and $\psi$ differ only by a possible renaming of variables carry over to all sentences $\psi \in \sln[\tau]$. Thus for Claim \ref{claim:secondslninduction} to hold, we only have to check whether $M$ is correct about $\mathcal B$'s satisfaction of $\sln$-formulas. But this is the case: Because $\lambda \in C^{(n)}$, $V_\lambda$ and $V$ agree about $\sln$-satisfaction (cf.\ Fact \ref{factSL}). Further, $\lambda$ is a limit point of $C^{(n)}$ and $M$ is correct about $C^{(n)}$ below $\lambda$ (as it is the image of an embedding witnessing $C^{(n)}$-strength). Thus $M$ sees that it is a limit point of $C^{(n)}$ and so $\lambda \in (C^{(n)})^M$. Therefore also $V_\lambda^M$ and $M$ agree about $\sln$-satisfaction. But also $V_\lambda^M = V_\lambda$ and therefore $V$, $V_\lambda$, and $M$ all agree on $\sln$-satisfaction. 
		
		\medskip Now let us start our induction to prove Claim \ref{claim:slninduction}. Recall that we do not allow sort quantifiers or negations to take infinitary formulas and so the base case and the cases of application of sort quantifiers and negation follow from Claim \ref{claim:secondslninduction}. The case for $\wedge$ is trivial.
		
		\medskip So let us consider the case $\psi = \exists x \chi $ and $\mathcal B \models \exists x \chi[f \circ (j \upharpoonright S)]$, where $x$ is a single (possibly second-order) variable. Then there is an $x$-variant $g$ of $f \circ (j \upharpoonright S)$ such that $\mathcal B \models \chi[g]$. Further $j(\psi) = \exists j(x) j(\chi)$. Fix $a = g(x)$. In $M$, we have to find a $j(x)$-variant $h$ of $f$ with $M \models ``\mathcal B \models j(\chi)[h]."$ Let $h$ be defined by
		\[
		h(y) = \begin{cases}
		a, &\text{ if } y = j(x)\\
		f(y), &\text{ otherwise.}
		\end{cases}
		\]
		For the case that $x$ is a second-order variable, note that $\mathcal B \in V_\lambda \subseteq M$ and so $M$ contains all subsets of $B$. Clearly, $h$ is a $j(x)$-variant of $f$.  Further, $h \in M$ as it was defined from $f \in M$. Also for $x \neq y \in S$, $j(y) \neq j(x)$, so $g(y) = f \circ (j \upharpoonright S)(y) = f(j(y)) =  h(j(y))$ and, by definition, $h(j(x)) = a = g(x)$. Hence $h \circ (j \upharpoonright S)  = g$. In particular $\mathcal B \models \chi[h \circ (j \upharpoonright S)]$ and thus, by induction hypothesis, $M \models ``\mathcal B \models j(\chi)[h]"$, as desired.

		\medskip If $\psi = \bigvee_{i < \gamma} \chi_i$ and $\mathcal B \models \psi[f \circ (j \upharpoonright S)]$, there is some $i < \gamma$ such that $\mathcal B \models \chi_i[f \circ (j \upharpoonright S)]$. By induction hypothesis, $M \models ``\mathcal B \models j(\chi_i)[f]."$ Now $j(\psi) = \bigvee_{k < j(\gamma)} \chi_k^*$ for some $\chi_k^*$. Because $\chi_i$ appears as a disjunct in $\psi$, $j(\chi_i)$ is one of the disjuncts of $j(\psi)$. In particular, then $M \models ``\mathcal B \models j(\psi)[f]"$. 
		
		\medskip If $\psi = \forall T \chi$ and $\mathcal B \models \forall T \chi[f \circ (j \upharpoonright S)]$. Then for any $T$-variant $g$ of $f \circ (j \upharpoonright S)$, we have $\mathcal B \models \chi[g]$. Note that $j(\psi) = \forall j(T) j(\chi)$. Now to show that $M \models ``B \models \forall j(T) j(\chi)[f]"$, we let $h$ be any $j(T)$-variant of $f$ with $h \in M$. This means $h \upharpoonright (j(S) \setminus j(T)) = f \upharpoonright (j(S) \setminus j(T))$. Consider $h' = h \circ (j \upharpoonright S)$. Then if $x \in (S \setminus T)$, we have $j(x) \in (j(S) \setminus j(T))$. And then $h'(x) = h(j(x)) = f(j(x))$, where the latter equality holds because $h$ is a $j(T)$-variant of $f$. But this shows that $h'$ is a $T$-variant of $f \circ (j \upharpoonright S)$. Thus $\mathcal B \models \chi[h']$, which just means $\mathcal B \models \chi[h \circ (j \upharpoonright S)]$. Then by induction hypothesis $M \models ``\mathcal B \models j(\chi)[h]"$. As $h$ was an arbitrary $j(T)$-variant of $f$ from $M$ we thus showed that $M \models  ``\mathcal B \models \forall j(T) j(\chi)[f]."$ This ends the proof of Claim \ref{ThatClaim}, and hence of (3) implies (2).
		
		\medskip Finally, we show that (2) implies (4). If $\kappa$ is the smallest $C^{(n)}$-strong cardinal, Claim \ref{ThatClaim} implies that $\kappa$ is an $\text{LS}^{\kappa}$ number of $\sln(\vee^\infty, \forall^\infty)$. Left to show is that $\kappa$ is minimal with that property. So suppose there is $\delta < \kappa$ which is an $\text{LS}^\delta$ number of $\sln(\vee^\infty, \forall^\infty)$. Then $\gamma = \text{LS}^\omega(\sln(\vee^\infty, \forall^\infty)) \leq \text{LS}^\delta(\sln(\vee^\infty, \forall^\infty)) \leq \delta < \kappa$. Because (3) implies (2), $\gamma$ is the smallest $C^{(n)}$-strong cardinal. But this contradicts that $\kappa$ is the smallest $C^{(n)}$-strong cardinal. \hfill \qed
	\end{Proof}
	
	\section{Characterising Shelah cardinals}\label{sec:Shelah}
	Finally, let us show that we can give the following compactness characterisation of \emph{Shelah cardinals}. To our best knowledge, this is the first known model-theoretic characterisation of Shelah cardinals. 
	\begin{theorem}\label{thm:ShelahCardinals}
		The following are equivalent for a cardinal $\kappa$:
		\begin{enumerate}
			\item[(1)] $\kappa$ is Shelah, i.e., for all $f: \kappa \rightarrow \kappa$, there is an elementary embedding $j: V \rightarrow M$ with $\crit(j) = \kappa$ and $V_{j(f)(\kappa)} \subseteq M$.
			\item[(2)] $\kappa$ is inaccessible and if $T \subseteq \mathcal L^2(\wedge^\infty, \exists^\infty, \vee^\kappa, \forall^\kappa)[\tau]$ is a theory of size $\kappa$ such that for every $\varphi \in T$ there is $\tau_\varphi \subseteq \tau$ such that $|\tau_\varphi| < \kappa$ and $\varphi \in \mathcal L^2(\wedge^\infty, \exists^\infty, \vee^\kappa, \forall^\kappa)[\tau_\varphi]$ and all $< \kappa$-sized subsets of $T$ have a model of size $< \kappa$, then $T$ has a model.
		\end{enumerate}
		\begin{proof}
			First assume (1). Clearly, a Shelah cardinal is measurable, so in particular inaccessible. To show the second part of (2), let $T = \{\varphi_i \colon i < \kappa \} \subseteq \mathcal L^2(\wedge^\infty, \exists^\infty, \vee^\kappa, \forall^\kappa)$ be a theory over a vocabulary $\tau$ such that every $\varphi_i$ only uses $< \kappa$ many symbols from $\tau$ and such that all its $< \kappa$-sized subsets are satisfiable by a model of size $< \kappa$. Then  writing $T_\alpha = \{\varphi_i \colon i < \alpha \}$ for $\alpha < \kappa$, every $T_\alpha$ is a theory of size $< \kappa$ over a vocabulary $\tau_\alpha$ of size $< \kappa$ which we can assume to be in $V_\kappa$. This implies that $|\tau| \leq \kappa$ and further we can write $\tau = \bigcup_{\alpha < \kappa} \tau_\alpha$ as an increasing union. Further, every $T_\alpha$ has a model $M_\alpha$ of size $< \kappa$ and hence without loss of generality we can assume that $M_\alpha \in V_\kappa$. Then we can let $f: \kappa \rightarrow \kappa$ be a function such that $M_\alpha \in V_{f(\alpha)}$. Because $\kappa$ is Shelah, there is $j: V \rightarrow M$ with $\crit(j) = \kappa$ and $V_{j(f)(\kappa)} \subseteq M$. Because Shelahness of $\kappa$ is witnessed by $\kappa$-complete extenders with length some $\lambda$ which is a strong limit cardinal of cofinality $\cof(\lambda) \geq \kappa$, we may without loss of generality assume that $M^{{<}\kappa} \subseteq M$. Then in $M$, $j(T)$ is of the form $\{\varphi_i^* \colon i < j(\kappa) \}$ for some formulas $\varphi_i^*$, and $T_\kappa^* = \{\varphi_i^* \colon i < \kappa \}$ is satisfiable by a model $\mathcal B \in V_{j(f)(\kappa)}$. Further $j `` T = \{j(\varphi_i) \colon i < \kappa\} \subseteq \{\varphi_i^* \colon i < \kappa \} = T_\kappa^*$, so $M \models ``\mathcal B \models j(\varphi)"$ for all $\varphi \in T$. Note that $j(\tau) \supseteq j``\tau = \tau$ and so $\mathcal B \upharpoonright \tau$ makes sense. We show that $\mathcal B \upharpoonright \tau \models T$ by showing the following claim:
			
			\begin{claim} For any formula $\varphi \in \mathcal L^2(\wedge^\infty, \exists^\infty, \vee^\kappa, \forall^\kappa)[\tau]$ with a set of free variables $S$ and any assignment $f: j(S) \rightarrow B \cup \bigcup_{n \in \omega} \mathcal P(B^n)$ in $M$, if $M \models ``\mathcal B \models j(\varphi)[f]"$, then $\mathcal B \upharpoonright \tau \models \varphi[f \circ j]$. 
			\end{claim}
			
			We show this by induction on the complexity of $\varphi$. If $\varphi = R(x_1, \dots, x_n)$ for some $R \in \tau$, then $R \in V_\kappa$ and so $j(\varphi) = R(j(x_1), \dots, j(x_n))$. Then the claim is clear by definition. 
			
			\medskip If $\varphi = \bigvee_{i < \gamma} \psi_i$ for some $\gamma < \kappa$, then $j(\varphi) = \bigvee_{i < \gamma} j(\psi_i)$. So if $\mathcal M \models ``\mathcal B \models j(\varphi)[f]"$, then $\mathcal M \models ``\mathcal B \models j(\psi_i)[f]"$ for some $i < \gamma$. Then by induction hypothesis $\mathcal B \upharpoonright \tau \models \psi_i[f \circ j]$, so also $\mathcal B \upharpoonright \tau \models \bigvee_{i <  \gamma} \psi_i[f \circ j]$. 
			
			\medskip
			
			If $\varphi = \forall Q \psi$ for some set of variables $Q$ of size $< \kappa$, note that because $M^{{<} \kappa} \subseteq M$, we have $j``Q = j(Q) \in M$. Then $j(\varphi) = \forall j``Qj(\psi)$. We assume $M \models ``\mathcal B \models \forall j``Qj(\psi)[f]"$. Let $g'$ be any $Q$-variant of $f \circ j$. We have to show $\mathcal B \upharpoonright \tau \models \psi[g']$. Define an assignment $g$ on $j(S)$ by letting
			\[
			g(x) =
			\begin{cases}
			g'(j^{-1}(x)), \text{ if } x \in j``Q \\
			f(x), \text{ if } x \in j(S) \setminus j``Q.
			\end{cases}
			\]
			Note that $g'(j^{-1}(x))$ makes sense for $x \in j``Q$ and that $g$ is a $j``Q$-variant of $f$. Because $M^{{<}\kappa} \subseteq M$ and $f \in M$, every $j``Q$-variant of $f$ belongs to $M$. In particular, $g \in M$ and so $\mathcal M \models `` \mathcal B \models j(\psi)[g]"$. Then by induction hypothesis, $\mathcal B \upharpoonright \tau \models \psi [g \circ j]$. Thus it is sufficient to show that $g \circ j = g'$. For $v \in Q$, we have $j(v) \in j``Q$ and so $g \circ j(v) = g(j(v)) = g'(j^{-1}(j(v))) = g'(v)$. And if $v \in S \setminus Q$, then $j(v) \in j(S) \setminus j``Q$, so $g(j(v)) = f(j(v)) = g'(v)$ where the latter holds because $g'$ is a $j``Q$-variant of $f \circ j$.
			
			\medskip If $\varphi = \bigwedge_{i < \delta} \psi_i$ for $\delta$ any ordinal, then $j(\varphi) = \bigwedge_{i < j(\delta)} \psi_i^*$ for some $\psi_i^*$. We assume $M \models ``\mathcal B \models \bigwedge_{i < j(\delta)} \psi_i^*[f]"$. Then if $k < \delta$, we have that $j(\psi_k)$ is among the $\psi_i^*$, so $M \models ``\mathcal B \models j(\psi_k)[f]"$. Then by induction hypothesis, $\mathcal B \upharpoonright \tau \models \psi_k[f \circ j]$. So overall, $\mathcal B \upharpoonright \tau \models \bigwedge_{i < \delta} \psi_i[f \circ j]$. 
			
			\medskip Finally, if $\varphi = \exists Q \psi$, where $Q$ is a set of variables of any size, then $j(\varphi) = \exists j(Q) j(\psi)$. We assume $M \models ``\exists j(Q) j(\psi)[f]"$. Then there is a $j(Q)$-variant $g$ of $f$ such that $M \models `` \mathcal B \models j(\psi)[g]".$ By induction hypothesis $\mathcal B \upharpoonright\tau \models \psi[g \circ j]$. Now if $x \in S \setminus Q$, then $j(x) \in j(S) \setminus j(Q)$, so $g(j(x)) = f(j(x))$. So $g \circ j$ is a $Q$-variant of $f \circ j$ with $\mathcal B \models \psi[g \circ j]$. Therefore $\mathcal B \models \exists Q \psi[f \circ j]$. 
			
			\medskip
			
			And now assume (2) and let $f: \kappa \rightarrow \kappa$. Without loss of generality, we can assume that $f$ is increasing. We want to produce $j: V \rightarrow M$ with $\crit(j) = \kappa$ and $V_{j(f)(\kappa)} \subseteq M$. For this purpose consider the sentence $\psi = \exists(X_A \colon A \in \mathcal P(V_\kappa^{< \omega}))(\bigwedge_{i = 1}^8 \psi_i)$ where $\psi_1$ to $\psi_6$ are as in the proof of Theorem \ref{thm:LS-Pin-strong} and where $\psi_7 = \forall x,y(F(x,y) \leftrightarrow X_f(x,y)) \wedge ``F \text{ is a total function}"$ with $F$ a new 2-place predicate symbol and $X_f$ the variable corresponding to $f \subseteq V_\kappa^2$. Further, take for $\alpha < \kappa$ a formula $\sigma_\alpha(x)$ of $\mathcal L_{\kappa \omega}$ that defines $\alpha$, i.e., with the property that if $M$ is transitive and $x \in M$, then $(M,\in) \models \sigma_\alpha(x)$ iff $x = \alpha$. Now let 
			\[
			\psi_8 = \bigwedge_{\alpha < \kappa} \forall x(X_\alpha(x)\leftrightarrow \bigvee_{\beta < \alpha} \sigma_\beta(x)).
			\]
			Then if $(M, \in, F^M) \models \psi$ is a transitive model, then there is a homomorphism of $\mathscr P$-structures $h: \mathscr P_{V_\kappa} \rightarrow \mathscr P_M$ such that $h(f) = F^M$ by letting $h(A) = X_A^M$ for $A \in \mathcal P(V_\kappa^{< \omega})$ where $(X_A^M \colon A \in \mathcal P(V_\kappa^{< \omega}))$ is the sequence witnessing that $(M, \in, F^M) \models \psi$. Equipped with $\psi$, take for every $\alpha \leq \kappa$ a new constant symbol $c_\alpha$ and consider the theory
			\[
			T = \{\psi\} \cup \{\Phi\} \cup \{``c_\alpha \text{ and } c_\beta \text{ are ordinals with } c_\alpha < c_\beta" \colon \alpha < \beta \leq \kappa \},
			\]
			where $\Phi$ is Magidor's $\Phi$ (cf.\ Fact \ref{fact:MagPhi}). Then any model of $T$ can without loss of generality be assumed to be of the form $(V_\delta, \in, f', c_\alpha^{V_\delta})_{\alpha \leq \kappa}$. Any model of $T$ has to contain an ordinal of order type $\geq \kappa$ by the last part of the theory and so $\delta > \kappa$. Since $(V_\delta, \in, f', c_\alpha^{V_\delta})_{\alpha \leq \kappa} \models \psi$, it gives rise to a homomorphism $h: \mathscr P_{V_\kappa} \rightarrow \mathscr P_{V_\delta}$ in the manner pointed out above. Note that by $\psi_7$ we have $h(f) = f'$.  By Fact \ref{fact:p-str}, this gives rise to an elementary embedding $j: V \rightarrow M$ such that $V_\delta \subseteq j(V_\kappa)$ and $h(A) = j(A) \cap V_\delta^{< \omega}$ for all $A \subseteq V_\kappa^{< \omega}$. In particular $f' = h(f) = j(f) \cap V_\delta^2$ and so $j(f)(\kappa) = f'(\kappa) < \delta$. Thus $V_{j(f)(\kappa)} \subseteq V_\delta \subseteq j(V_\kappa) \subseteq M$. So we only have to show that $\crit(j) = \kappa$ to see that $\kappa$ is Shelah. By our choice of $T$, $V_\delta$ contains an ordinal of order type $\kappa + 1$ and so we have $\delta > \kappa$. Because $V_\delta \subseteq j(V_\kappa)$ we have to have $\crit(j) \leq \kappa$. And further, for $\alpha < \kappa$, we have by usage of $\psi_8$ that $\alpha = h(\alpha) = j(\alpha) \cap V_\delta$ and so $\crit(j) = \kappa$. 
			
			\medskip Left to show is that $T$ is satisfiable. Clearly $T$ is of size $\kappa$ and every sentence in $T$ only uses finitely many symbols from $T$'s vocabulary. So by our assumption it is sufficient to show that every $< \kappa$ sized subset $T_0$ of $T$ has a model of size $< \kappa$. So let $T_0$ be such a subset and $\gamma = \sup\{\beta <  \kappa \colon ``c_\alpha < c_\beta" \in T_0 \} < \kappa$ and take any ordinal $\eta$ with $\gamma < \eta < \kappa$ such that $\eta$ is closed under $f$. Note that such an $\eta$ exists because $\kappa$ is regular. Then let $c_\alpha^{V_\eta} = \alpha$ for $\alpha \leq \gamma$ and $c_\kappa^{V_{\eta}} = \gamma + 1$. Then $(V_\eta, \in, f \upharpoonright \eta, c_\alpha^{V_\eta},c_\kappa^{V_\eta} )_{\alpha \leq \gamma} \models T_0$, because by Fact \ref{fact:trivialHom}, $V_\eta \subseteq V_\kappa$ implies that $A \mapsto A \cap V_\eta^{< \omega}$ for $A \subseteq V_\kappa^{< \omega}$ defines a (trivial) homomorphism $\mathscr P_{V_\kappa} \rightarrow \mathscr P_{V_\eta}$ and the rest of $T_0$ is satisfied by our choice of the constants. Finally, $|V_\eta| < \kappa$, as $\kappa$ is inaccessible.
		\end{proof}
	\end{theorem}

	\bibliography{bibliography}{}
	\bibliographystyle{alpha}

\end{document}